\documentclass[10pt]{amsart}
\usepackage{times,amsmath,amsbsy,amssymb,amscd,mathrsfs}
\usepackage{graphicx,epstopdf,wrapfig,chemarrow}

\usepackage{algorithm2e} 
\usepackage{multicol,multirow}
\usepackage{mathtools}
\usepackage[usenames,dvipsnames,svgnames,table]{xcolor}
\usepackage[all]{xy}
\usepackage{wrapfig}
\usepackage{tcolorbox}

\usepackage{tikz,tikz-cd}
\usepackage[utf8]{inputenc}
\usepackage{pgfplots} 
\usepackage{pgfgantt}
\usepackage{pdflscape}
\pgfplotsset{compat=newest} 
\pgfplotsset{plot coordinates/math parser=false}
\newlength\fwidth

\definecolor{myBlue}{rgb}{0.0,0.0,0.55}
\usepackage[pdftex,colorlinks=true,citecolor=myBlue,linkcolor=myBlue]{hyperref}

\usepackage[hyperpageref]{backref}

\usepackage{comment,enumerate,multicol,xspace}
\usepackage{enumitem}

  \newcounter{mnote}
  \let\oldmarginpar\marginpar
    \renewcommand\marginpar[1]{\-\oldmarginpar[\raggedleft\footnotesize #1]%
    {\raggedright\footnotesize #1}}

\newtheorem{theorem}{Theorem}[section]
\newtheorem{lemma}[theorem]{Lemma}

\newtheorem{example}[theorem]{Example}

\newtheorem{remark}[theorem]{Remark}

\newcommand{\dx}{\,{\rm d}x}

\newcommand{\bs}{\boldsymbol}

\newcommand{\curl}{{\rm curl\,}}
\renewcommand{\div}{\operatorname{div}}
\newcommand{\grad}{{\rm grad\,}}
\newcommand{\tr}{\operatorname{tr}}
\newcommand{\dev}{\operatorname{dev}}

\newcommand{\skw}{\operatorname{skw}}

\newcommand{\vertiii}[1]{{\left\vert\kern-0.25ex\left\vert\kern-0.25ex\left\vert #1 
    \right\vert\kern-0.25ex\right\vert\kern-0.25ex\right\vert}}

\usepackage{booktabs}
\usepackage{stmaryrd}
\usepackage{scalefnt}
\usepackage{graphicx}
\usepackage{subcaption}
\allowdisplaybreaks[1]

\begin{document}
\title[Superconvergent and Divergence-Free FEM for Stokes]{Superconvergent and Divergence-Free Mixed Finite Element Methods for The Stokes Equation}
\author{Long Chen}
\address{Department of Mathematics, University of California at Irvine, Irvine, CA 92697, USA}
\email{chenlong@math.uci.edu}
\author{Xuehai Huang}%
\address{School of Mathematics, Shanghai University of Finance and Economics, Shanghai 200433, China}%
\email{huang.xuehai@sufe.edu.cn}%
\author{Chao Zhang}%
\address{School of General Education, Wenzhou Business College, Wenzhou 325035, China}%
\email{zhang.chao@wzbc.edu.cn}%
\author{Xinyue Zhao}%
\address{School of Mathematics, Shanghai University of Finance and Economics, Shanghai 200433, China}%
\email{zhaoxinyue20210921@163.com}%

\thanks{The first author was supported by NSF DMS-2309777 and DMS-2309785. The second author was supported by the National Natural Science Foundation of China Project 12171300.}

\makeatletter
\@namedef{subjclassname@2020}{\textup{2020} Mathematics Subject Classification}
\makeatother

\begin{abstract}
This paper develops divergence-free mixed finite element methods for the Stokes equation. Using H(div)-conforming velocities and discontinuous pressures ensures the inf-sup condition for the velocity--pressure pair and yields pointwise divergence-free velocities. However, this choice makes the vector Laplacian difficult to discretize. Inspired by mass-conserving mixed formulations with stresses, tangential--normal continuous traceless tensor elements are used to discretize the vector Laplacian. An inf-sup condition for the weak div operator between the stress and velocity spaces is then proved. Two key properties characterize the scheme. First, the stress--velocity inf-sup stability gives a stable discretization of the vector Laplacian without additional stabilization, unlike discontinuous Galerkin or virtual element methods. Second, the scheme has the property that if a stress field is distributionally divergence-free against the discrete divergence-free velocity space, then it is also distributionally divergence-free against the continuous divergence-free velocity space. This property decouples the stress and velocity errors and leads to superconvergence. As a result, optimal-order error estimates are obtained for the stress, while the discrete velocity is superclose to its H(div) interpolant. The projected-pressure error estimate is optimal with Raviart-Thomas velocities and superconvergent with Brezzi-Douglas-Marini velocities, while local postprocessing yields an elementwise divergence-free velocity with higher-order convergence. Numerical experiments confirm the theoretical results.
\end{abstract}

\maketitle


\section{Introduction}

In this work, we develop mixed finite element methods for the Stokes equations on a bounded domain $\Omega \subset \mathbb{R}^d$ ($d \geq 2$) with a given body force $\boldsymbol{f} \in L^{2}(\Omega; \mathbb{R}^d)$:
\begin{equation}\label{eq:stokes}
\begin{aligned}
    -\Delta \boldsymbol{u} - \nabla p &= \boldsymbol{f} && \text{in } \Omega, \\
    \operatorname{div} \boldsymbol{u} &= 0 && \text{in } \Omega, \\
    \boldsymbol{u} &= \boldsymbol{0} && \text{on } \partial \Omega,
\end{aligned}
\end{equation}
where $\boldsymbol{u} \in H_0^1(\Omega; \mathbb{R}^d)$ and $p \in L^2_0(\Omega)$ represent the velocity and pressure, respectively.

Standard velocity--pressure finite element schemes for the Stokes equations depend crucially on the discrete divergence operator induced by the bilinear form $(\operatorname{div} \boldsymbol{v}_h, q_h)$. A stable scheme requires the discrete spaces $\boldsymbol{V}_h$ and $P_h$ to satisfy the inf-sup (or LBB) condition \cite{Brezzi1974}:
\begin{equation}\label{intro:div}
\inf_{q_h \in P_h \setminus \{0\}} \sup_{\boldsymbol{v}_h \in \boldsymbol{V}_h \setminus \{\boldsymbol{0}\}} \frac{(\operatorname{div} \boldsymbol{v}_h, q_h)}{\|\boldsymbol{v}_h\|_{\boldsymbol{V}_h} \|q_h\|_{P_h}} \geq \beta > 0,
\end{equation}
where $\beta$ is a constant independent of the mesh size. Standard stable pairs for \eqref{eq:stokes} include the Taylor--Hood element \cite{TaylorHood1973}, the MINI element \cite{ArnoldBrezziFortin1984}, and the nonconforming $P_1$--$P_0$ element \cite{CrouzeixRaviart1973}; a comprehensive review is available in \cite{BoffiBrezziFortin2013}.  However, these elements do not enforce the divergence-free condition pointwise. Losing exact mass conservation often leads to poor pressure robustness in the velocity error estimates. We refer to \cite{JohnLinkeMerdonNeilanEtAl2017} for a detailed discussion on the importance of strongly divergence-free constraints.

Enforcing the discrete divergence-free condition, $(\operatorname{div} \boldsymbol{u}_h, q_h) = 0$ for all $q_h \in P_h$, does not yield a strongly divergence-free velocity field (i.e., $\operatorname{div} \boldsymbol{u}_h = 0$ pointwise) because $P_h$ is only a proper subspace of $L^2_0(\Omega)$. Schemes that achieve this pointwise property are called strongly divergence-free. In such schemes, one can choose an appropriate interpolant $\boldsymbol{u}_I$ such that $\operatorname{div} (\boldsymbol{u}_I - \boldsymbol{u}_h) = 0$ to decouple the velocity error from the pressure error via Galerkin orthogonality:
\begin{equation}\label{eq:probust}
(\nabla (\boldsymbol{u} - \boldsymbol{u}_h), \nabla (\boldsymbol{u}_I - \boldsymbol{u}_h)) = -(p - p_h, \operatorname{div} (\boldsymbol{u}_I - \boldsymbol{u}_h)) = 0.
\end{equation}
If the scheme is not strongly divergence-free, the term $(p, \operatorname{div} (\boldsymbol{u}_I - \boldsymbol{u}_h))$ does not generally vanish, introducing a pressure dependence into the velocity error. When it does vanish, we obtain the bound $\|\nabla (\boldsymbol{u} - \boldsymbol{u}_h)\| \leq \|\nabla (\boldsymbol{u} - \boldsymbol{u}_I)\|$, which depends only on the regularity of the velocity not the pressure. This decoupling yields a pressure-robust error estimate.

Several methods provide strongly divergence-free velocity approximations. These include the Scott--Vogelius element \cite{ScottVogelius1985} (which requires specific mesh conditions), smooth finite elements \cite{Zhang2011,ChenHuang2024a,ChenHuang2025}, conforming pairs on split meshes \cite{ArnoldQin1992a, Zhang2005, ChristiansenHu2018, FuGuzmanNeilan2020, GuzmanLischkeNeilan2020, HuZhangZhang2022}, and conforming rational functions \cite{GuzmanNeilan2014,GuzmanNeilan2014a}. Low-order nonconforming pairs \cite{MardalTaiWinther2002,TaiWinther2006,XieXuXue2008} are also used. 

A natural way to enforce the strongly divergence-free condition is to employ $H(\operatorname{div})$-conforming velocity spaces together with discontinuous polynomial pressure spaces. For example, using the Raviart--Thomas ($\mathrm{RT}_k$) \cite{RaviartThomas1977,Nedelec1980} or Brezzi--Douglas--Marini ($\mathrm{BDM}_k$) \cite{BrezziDouglasMarini1985,Nedelec1986} spaces for velocity (with $\boldsymbol{u}\cdot \boldsymbol{n} = 0$ on $\partial \Omega$) paired with the discontinuous space $\mathbb{P}_{\ell}(\mathcal{T}_h)$ for pressure ensures LBB stability and yields a strongly divergence-free velocity field. 

However, because the discrete velocity space is not a subspace of $H^1(\Omega;\mathbb{R}^d)$, the vector Laplacian $-\Delta \boldsymbol{u}$ lacks a standard $H^1$ weak formulation. 

To address this, the vorticity--velocity--pressure formulation \cite{DuboisSalauenSalmon2003a,DuboisSalauenSalmon2003} introduces the vorticity $\boldsymbol{\omega} = \operatorname{curl} \boldsymbol{u}$ and treats the Laplacian as the Hodge Laplacian: $-\Delta \boldsymbol{u} = \operatorname{curl} \operatorname{curl} \boldsymbol{u} - \nabla \operatorname{div} \boldsymbol{u}$. Yet, when enforcing no-slip boundary conditions ($\boldsymbol{u} = \boldsymbol{0}$ on $\partial \Omega$), this formulation suffers from reduced stability and suboptimal convergence rates due to the absence of a proper Hilbert complex \cite{ArnoldFalkGopalakrishnan2012}. Symmetrical meshes can partially recover these rates, as proved in \cite{ChenWangZhong2015} for triangular MAC (TMAC) schemes, which lump the mass matrix to eliminate vorticity. 

Alternatively, the mass-conserving mixed stress (MCS) formulation \cite{GopalakrishnanLedererSchoeberl2020,GopalakrishnanLedererSchoeberl2020a} embeds $H(\operatorname{div})$ velocities in a pseudostress--velocity--pressure system. The finite element space for the stress is tangential--normal continuous. While standard stress formulations yield $\mathcal{O}(h^k)$ convergence \cite{GopalakrishnanLedererSchoeberl2020a}, enriching the stress space allows the $\mathrm{RT}_k$--$\mathrm{P}_k$ pair to achieve $\mathcal{O}(h^{k+1})$ convergence, with requirement $k \ge 1$ \cite{GopalakrishnanLedererSchoeberl2020}.
A sophisticated inf-sup condition is proved for tangential--normal continuous traceless tensors.

Inspired by the MCS framework, we use tangential--normal continuous elements, $\Sigma_{k}^{\operatorname{tn}}$, to discretize the vector Laplacian. Because $\Sigma_{k}^{\operatorname{tn}}$ is not $H(\operatorname{div})$-conforming, we define a weak divergence operator, $\operatorname{div}_w$, whose adjoint is the weak deviatoric gradient, $\operatorname{dev}\operatorname{grad}_w$, acting on the velocity element.

The viability of this method depends on establishing an inf-sup condition for $\operatorname{div}_w$ over the stress--velocity pair. This presents a structural tension: the velocity space must be large enough (e.g., $\mathrm{RT}_k$) to maintain the strongly divergence-free property of $\div_w$, but small enough (e.g., $\mathrm{BDM}_k$) to ensure that $\operatorname{div}_w \Sigma_{k}^{\operatorname{tn}}$ remains surjective. We resolve this conflict by establishing the inf-sup condition directly on the divergence-free subspace:
\begin{equation}\label{intro:divw}
\inf_{\boldsymbol{v}_h\in \mathrm{BDM}_k\cap \ker(\operatorname{div})} 
\sup_{\boldsymbol{\tau}_h\in \Sigma_{k}^{\operatorname{tn}}} 
\frac{(\operatorname{div}_w \boldsymbol{\tau}_h, \boldsymbol{v}_h)_{0,h}}
{\|\boldsymbol{\tau}_h\|_{\operatorname{div}_w}\,\|\boldsymbol{v}_h\|} 
= \alpha > 0,
\end{equation}
leveraging the crucial algebraic identity $\mathrm{BDM}_k\cap \ker(\operatorname{div}) = \mathrm{RT}_k\cap \ker(\operatorname{div})$ as space $\mathrm{RT}_k$ is an enrichment of $\mathrm{BDM}_k$ to enlarge only the range of the divergence operator.
Furthermore, we establish the following distributional divergence-free property against divergence-free velocity fields: if
\[
(\operatorname{div}_w \boldsymbol{\tau}_h, \boldsymbol{v}_h)_{0,h} = 0
\qquad \forall\, \boldsymbol{v}_h \in \mathrm{BDM}_k \cap \ker(\operatorname{div}),
\]
then
\[
\langle \operatorname{div} \boldsymbol{\tau}_h, \boldsymbol{v}\rangle = 0
\qquad \forall\, \boldsymbol{v} \in H_0^1(\Omega;\mathbb R^d)\cap \ker(\operatorname{div}).
\]
This is sufficient to decouple the stress and velocity approximation errors.

The constraint $\div \boldsymbol{u}_h = 0$ is imposed by introducing a Lagrange multiplier, which has the physical interpretation of pressure. Specifically, let $\mathring{\mathbb{V}}_{k,\ell}^{\operatorname{div}}$ denote the $\mathrm{RT}_k$ space when $\ell=k$, or the $\mathrm{BDM}_k$ space when $\ell=k-1$, satisfying the boundary condition $\boldsymbol{u}\cdot \boldsymbol{n} = 0$ on $\partial \Omega$. The discrete problem seeks $(\boldsymbol{u}_h, p_h) \in \mathring{\mathbb{V}}_{k,\ell}^{\operatorname{div}} \times (\mathbb{P}_{\ell}(\mathcal{T}_h)/\mathbb R)$ for $k\geq 0$ and $\ell \in \{k, k - 1\}$ such that:
\begin{equation}\label{eq:intro-scheme}
\begin{cases}
a_h(\boldsymbol{u}_h, \boldsymbol{v}_h) + b(\boldsymbol{v}_h, p_h) = (\boldsymbol{f}, \boldsymbol{v}_h) & \text{for all } \boldsymbol{v}_h \in \mathring{\mathbb{V}}_{k,\ell}^{\operatorname{div}}, \\
b(\boldsymbol{u}_h, q_h) = 0 & \text{for all } q_h \in \mathbb{P}_{\ell}(\mathcal{T}_h)/\mathbb R,
\end{cases}
\end{equation}
where $a_h(\boldsymbol{u}_h, \boldsymbol{v}_h) := (\operatorname{dev}\operatorname{grad}_w \boldsymbol{u}_h, \operatorname{dev}\operatorname{grad}_w \boldsymbol{v}_h)$ and $b(\boldsymbol{v}_h, q_h) := (\operatorname{div} \boldsymbol{v}_h, q_h)$. Section~\ref{sec:vectorLap} gives the precise definitions of these spaces and of the operator $\operatorname{dev}\operatorname{grad}_w$, while Section~\ref{sec:stokesmfem} presents a hybridized form of \eqref{eq:intro-scheme}. The tangential boundary condition $\Pi_F\boldsymbol{u}=0$ on $\partial\Omega$ is imposed weakly and may be replaced by $\Pi_F\boldsymbol{\sigma}\boldsymbol{n}=0$ on $\partial\Omega$, where the stress $\boldsymbol{\sigma} = \operatorname{dev} \operatorname{grad} \boldsymbol{u}$.

We emphasize that this method requires no additional stabilization. The discrete bilinear form $a_h(\cdot,\cdot)$ is not augmented by mesh-dependent penalty parameters or grad-div stabilization terms. Instead, stability follows directly from the discrete inf-sup conditions for the $\operatorname{div}_w$ \eqref{intro:divw} and $\operatorname{div}$ operators \eqref{intro:div}. The weak deviatoric gradient, $\operatorname{dev}\operatorname{grad}_w$, naturally incorporates face contributions via elementwise integration by parts, yielding a consistent, parameter-free discretization that recovers the continuous identity for smooth functions.

Let $\boldsymbol{\sigma}_h = \operatorname{dev} \operatorname{grad}_w \boldsymbol{u}_h \in \Sigma_{k}^{\operatorname{tn}}$ denote the discrete stress. Owing to the strongly divergence-free property of $\operatorname{div}_w$, the stress error decouples from the velocity and pressure (using an argument similar to \eqref{eq:probust}), achieving optimal $\mathcal{O}(h^{k+1})$ convergence in $L^2$. This yields the following estimate:
\begin{equation*}
\|\boldsymbol{\sigma} - \boldsymbol{\sigma}_h\| + \|\operatorname{dev}\operatorname{grad}_w (\boldsymbol{u}_I - \boldsymbol{u}_h)\| + \|Q_{\ell} p - p_h\| \lesssim h^{k+1} |\boldsymbol{\sigma}|_{k+1},
\end{equation*}
where $\boldsymbol{u}_I$ is an interpolant of $\boldsymbol{u}$ into $\mathring{\mathbb{V}}_{k,\ell}^{\operatorname{div}}$, and $Q_\ell$ is the $L^2$-orthogonal projection onto $\mathbb{P}_\ell(\mathcal{T}_h)$. Because $h^{k+1}$ exceeds the standard approximation capabilities of the chosen velocity and pressure spaces, both the velocity error $\|\operatorname{dev}\operatorname{grad}_w (\boldsymbol{u}_I - \boldsymbol{u}_h)\|$ and the pressure error $\|Q_{\ell} p - p_h\|$ are superconvergent. This structural superconvergence does not require mesh symmetry—unlike TMAC \cite{ChenWangZhong2015}—and it facilitates higher-order velocity post-processing.

Alternatively, discontinuous Galerkin (DG) methods can discretize the vector Laplacian \cite{WangYe2007,ChenFengXie2016,KimChungLee2013,CockburnSayas2014}. However, the DG method in \cite{WangYe2007} requires a penalty term for stability, and the hybridizable DG (HDG) method in \cite{CockburnSayas2014} uses $k$th-order polynomials for all variables, whereas for $\boldsymbol{u}_h \in \mathrm{BDM}_k$, it suffices to choose $p_h \in \mathbb{P}_{k-1}(\mathcal{T}_h)$. Similarly, the divergence-free weak virtual element method (VEM) \cite{ChenWang2019,BeiraodaVeigaLovadinaVacca2017,WeiHuangLi2021,HuangWang2023} generally requires stabilization for the projected discrete gradient and relies on the full $\mathrm{BDM}_k$ space; using the incomplete $\mathrm{RT}_k$ space degrades convergence. Consequently, their energy-norm velocity estimates are optimal order, but not superconvergent.

Our framework extends to the lowest-order case $\ell = 0$, which the MCS framework does not address \cite{GopalakrishnanLedererSchoeberl2020,GopalakrishnanLedererSchoeberl2020a}. Specifically, the $(k, \ell) = (0, 0)$ configuration (illustrated in Fig. \ref{fig:lowestorder}) generalizes the classical MAC scheme \cite{HarlowWelchothers1965Numerical} to unstructured triangulations. It achieves the first-order convergence without the accuracy loss typical of vorticity-based formulations. Similarly, the $\mathrm{BDM}_1$--$\mathbb{P}_0$ pair, $(k, \ell) = (1, 0)$, provides a second-order scheme.

\begin{figure}[htbp]
\begin{center}
\includegraphics[width=3.1in]{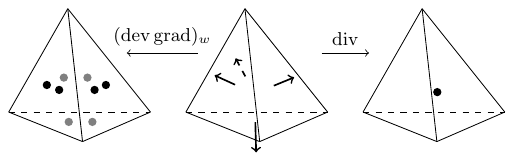}
\caption{Lowest-order configuration $(k, \ell) = (0, 0)$. The velocity--pressure pair is ${\rm RT}_0$--$\mathbb{P}_0$, and the stress is a piecewise constant traceless matrix with tangential--normal continuity.}
\label{fig:lowestorder}
\end{center}
\end{figure}

We can further hybridize \eqref{eq:intro-scheme} by relaxing the normal continuity of the velocity, which permits the choice $(k,\ell)=(0,-1)$. Here, the velocity is approximated by piecewise constant vectors, and normal continuity is enforced via a nonconforming linear pressure space. This configuration maintains the first-order convergence.

By using hybridization to enforce the tangential-normal continuity of the stress, the $(k, \ell) = (0, 0)$ case becomes equivalent to the Crouzeix--Raviart (CR)--$\mathbb{P}_0$ pair \cite[(23)]{Linke2014}, with one crucial modification: in the source term $(\boldsymbol{f}, \boldsymbol{v}_h)$, the test function $\boldsymbol{v}_h$ is locally reconstructed as an $\mathrm{RT}_0$ field rather than a standard CR basis function. While the classical CR--$\mathbb{P}_0$ method is not pressure-robust, this $\mathrm{RT}_0$-based reconstruction recovers pressure robustness while preserving the first-order convergence.

Although $H(\operatorname{div})$-conforming velocity fields naturally enforce mass conservation, the primary challenge remains the consistent discretization of the vector Laplacian. We address this by developing a distributional framework in Section~\ref{sec:vectorLap} and establishing the inf-sup condition for the stress--velocity pairs. In Section~\ref{sec:stokesmfem}, we formulate the mixed finite element methods for the Stokes equations by introducing the pressure as a Lagrange multiplier to enforce the divergence-free constraint. We prove the stability of $\operatorname{div}_w$ onto $\mathring{\mathbb V}_{k,k-1}^{\operatorname{div}}$, including the divergence-free subspace $\mathring{\mathbb{V}}_{k,\ell}^{\operatorname{div}} \cap \ker(\operatorname{div})$, and establish optimal-order error estimates. Furthermore, Section~\ref{sec:hybrid} explores equivalent discrete formulations, including virtual element and pseudostress methods. Finally, Section~\ref{sec:numericalresults} provides numerical experiments to verify our theoretical findings, and Section~\ref{sec:conclusion} offers concluding remarks.

\section{Distributional Discretization of the Vector Laplacian}\label{sec:vectorLap}
The Stokes equation can be viewed as a vector Poisson problem on the divergence-free subspace, i.e., $-\Delta \boldsymbol{u} = \boldsymbol{f}$ for $\boldsymbol{u}\in H_0^1(\Omega; \mathbb{R}^d)\cap \ker(\div).$ However, $H(\div)$ velocity fields do not have the $H^1$ regularity needed for a classical Galerkin treatment of the Laplacian. This section develops the distributional framework needed to address this issue.

By rewriting the Laplacian as a first-order system with a traceless tensor (pseudostress), weak divergence and weak deviatoric gradient operators, $\mathrm{div}_w$ and $\dev\grad_w$, are introduced. Their stability and commuting properties are established to ensure consistency with the underlying integration-by-parts identities. With this mixed approximation of the vector Laplacian in place, the full Stokes system in Section~\ref{sec:stokesmfem} follows by enforcing the divergence-free constraint through a pressure Lagrange multiplier.

\subsection{Notation}
Let $\Omega \subset \mathbb{R}^d$ ($d \ge 2$) be a bounded polytope with boundary $\partial\Omega$. For a bounded domain $D$ and an integer $m \ge 0$, we denote by $H^m(D)$ the standard Sobolev space with norm $\|\cdot\|_{m,D}$ and semi-norm $|\cdot|_{m,D}$. Let $H_0^m(D)$ be the closure of $C_0^\infty(D)$ in $H^m(D)$. We set $L^2(D) = H^0(D)$ with inner product $(\cdot,\cdot)_D$ and norm $\|\cdot\|_D$. When $D=\Omega$, the subscript $D$ is omitted. For any $D$, $h_D$ denotes its diameter and $\boldsymbol{n}_{\partial D}$ its unit outward normal vector. It will be abbreviated as $\boldsymbol{n}$, when the domain $D$ is clear from the context.

We define the following spaces for $H(\div)$:
$$
\begin{aligned}
H(\div,D) &=\{\boldsymbol{v}\in L^2(D;\mathbb{R}^d):\div\boldsymbol{v}\in L^2(D)\},\\
H_0(\div,D)&=\{\boldsymbol{v}\in H(\div,D):\boldsymbol{v}\cdot\boldsymbol{n}=0 \text{ on }\partial D\},
\end{aligned}
$$
endowed with the norm $\|\boldsymbol{v}\|_{\div} = (\|\boldsymbol{v}\|^2 + \|\div \boldsymbol{v}\|^2)^{1/2}$.
Let $L_0^2(D)=\{q\in L^2(D):\int_D q\,\dx=0\}$. For any subspace $V\subseteq L^2(D)$, we denote $V/\mathbb{R}=V\cap L_0^2(D)$.

Let $\{\mathcal{T}_h\}_{h>0}$ be a regular family of simplicial meshes of $\Omega$, with $h = \max_{T \in \mathcal{T}_h} h_T$ and $h_T = \operatorname{diam}(T)$. We denote by $\mathcal{F}_h$ and $\mathring{\mathcal{F}}_h$ the sets of all faces and interior faces, respectively. 
For any $T\in\mathcal{T}_h$, $\mathcal{F}(T)$ denotes the set of faces of $T$.

For each interior face $F = T^+ \cap T^-$, we fix a unit normal $\boldsymbol{n}_F$ to orient jumps and averages. Recall that $\boldsymbol{n}_{\partial T}$ is the outward unit normal on the boundary of an element $T$. On an interior face $F$, we set $\boldsymbol{n}_F=\boldsymbol{n}_{\partial T^+}|_F$ and thus $\boldsymbol{n}_{\partial T^-}|_F = -\boldsymbol{n}_F$. For boundary faces $F \subset \partial\Omega$, we set $\boldsymbol{n}_F = \boldsymbol{n}_{\partial\Omega}|_F$.
For two adjacent elements $T^{\pm}$ sharing a face $F$, the jump of a function $v$ is defined as
$$ \llbracket v \rrbracket|_F = v^+ \boldsymbol{n}_F \cdot \boldsymbol{n}_{\partial T^{+}} + v^- \boldsymbol{n}_F \cdot \boldsymbol{n}_{\partial T^{-}}, $$
and on boundary faces $F \subset \partial\Omega$, we set $\llbracket v \rrbracket = v|_F$.

For a face $F$ and a vector $\boldsymbol{v}\in\mathbb{R}^d$, the tangential projection is
$$
\Pi_F\boldsymbol{v} = (\boldsymbol{I} - \boldsymbol{n}_F\boldsymbol{n}_F^{\intercal})\boldsymbol{v}.
$$
For each face $F$, let $\{\boldsymbol{t}_{i}\}_{i=1}^{d-1}$ be an orthonormal basis of the tangent space of $F$. We identify a tangential vector $\boldsymbol{v}_t$ with its coordinate vector $(\boldsymbol{t}_{i}^{\intercal}\boldsymbol{v}_t)_{i=1}^{d-1}\in\mathbb R^{d-1}$. Under this identification, $\Pi_F\boldsymbol{v}$ is also regarded as an element of $\mathbb R^{d-1}$.

For a domain $D$ and integer $k\ge0$, let $\mathbb{P}_k(D)$ be the space of polynomials on $D$ of degree $\le k$, $\mathbb{H}_k(D)$ the space of homogeneous polynomials of degree $k$, and $Q_{k,D}$ the $L^2$-projection onto $\mathbb{P}_k(D)$. We understand $\mathbb P_{-1}(D) = \mathbb H_{-1}(D) = \{0\}$.
Set $\mathbb{M}=\mathbb{R}^{d\times d}$, and define $\mathbb{K}$ and $\mathbb{T}$ to be the skew-symmetric and
traceless subspaces of $\mathbb{M}$, respectively. For $\boldsymbol{\tau}\in\mathbb{M}$, define
\[
\operatorname{dev}\boldsymbol{\tau}
=\boldsymbol{\tau}-\tfrac{1}{d}(\operatorname{tr}\boldsymbol{\tau})\boldsymbol{I}\in\mathbb{T},
\qquad
\operatorname{skw}\boldsymbol{\tau}
=\tfrac{1}{2}\bigl(\boldsymbol{\tau}-\boldsymbol{\tau}^{\intercal}\bigr)\in\mathbb{K}.
\]

Define
\begin{align*}
&H^s(\mathcal{T}_h) = \{v\in L^2(\Omega): v|_T\in H^s(T)\;\textrm{ for all } T\in \mathcal T_h\} = \prod_{T\in \mathcal T_h} H^s(T),\\ 
&\mathbb{P}_k(\mathcal{T}_h) = \{v\in L^2(\Omega): v|_T\in\mathbb{P}_k(T)\;\textrm{ for all } T\in \mathcal T_h\} = \prod_{T\in \mathcal T_h} \mathbb P_{k}(T),\\
&\mathbb{P}_k(\mathcal{F}_h) = \{v\in L^2(\mathcal{F}_h): v|_F\in\mathbb{P}_k(F)\;\textrm{ for all } F\in \mathcal F_h\} = \prod_{F\in \mathcal F_h} \mathbb P_{k}(F),
\end{align*}
and similarly define $\mathbb{P}_k(\mathring{\mathcal{F}}_h)$ for interior faces only.
For a linear space $V(D)$ or $V(\mathcal{T}_h)$, we denote its vector- or tensor-valued version by
$$
V(D; \mathbb{X}) = V(D) \otimes \mathbb{X}, \quad V(\mathcal{T}_h; \mathbb{X}) = V(\mathcal{T}_h) \otimes \mathbb{X},
$$
where $\mathbb{X} \in \{\mathbb{R}^d, \mathbb{M}, \mathbb{T}, \mathbb{K}\}$.


Let $Q_{k}$ be the $L^2$-projection onto $\mathbb{P}_k(\mathcal{T}_h)$, or its appropriate vector- or tensor-valued version. We denote the element-wise gradient and divergence operators by $\nabla_h$ (or $\grad_h$ for vectors) and $\div_h$, respectively.

The main discrete spaces, operators, and interpolations are summarized in Table~\ref{tab:spaces-ops} for ease of reference.
\begin{table}[htbp]
\centering
\caption{Summary of main discrete spaces, operators, and interpolation maps.}
\label{tab:spaces-ops}
\renewcommand{\arraystretch}{1.25}
\small
\begin{tabular}{lll}
\toprule
\textbf{Object} & \textbf{Description} & \textbf{Role / Mapping} \\
\midrule
$\mathring{\mathbb{V}}^{\div}_{k,\ell}(\mathcal{T}_h)$ & $H(\div)$-conforming space & Discrete velocity \\
$\mathbb{P}_{\ell}(\mathcal{T}_h)$      & Discontinuous polynomials  & Discrete pressure \\
$\Sigma^{\rm tn}_{k}(\mathcal{T}_h)$                   & t--n continuous traceless tensors & Discrete pseudo-stress \\
$\mathring{M}^{-1}_{k-1,k}(\mathbb R^d)$ &
Vector-valued broken pair & $\mathbb{P}_{k-1}(\mathcal{T}_h; \mathbb R^d) \times \mathbb{P}_k(\mathring{\mathcal{F}}_h; \mathbb R^d)$\\
$\div_w$                                & Weak divergence            & $\div_w: \Sigma^{\rm tn}_{k} \to \mathring{\mathbb{V}}^{\div}_{k,\ell}$ \\
$\dev\grad_w$                           & Weak deviatoric gradient   & $\dev\grad_w: \mathring{\mathbb{V}}^{\div}_{k,\ell} \to \Sigma^{\rm tn}_{k}$ \\
$I^{\rm tn}_{k}$                        & t--n interpolant            & $I^{\rm tn}_{k}: H^1 \to \Sigma^{\rm tn}_{k}$ \\
$I^{\div}_{k,\ell}$                     & $H(\div)$ interpolant       & $I^{\div}_{k,\ell}: H^1 \to \mathring{\mathbb{V}}^{\div}_{k,\ell}$ \\
\bottomrule
\end{tabular}
\end{table}

\subsection{Weak div stability for the scalar Laplacian}
We first consider the surjective map
$$
\operatorname{div}: L^{2}(\Omega;\mathbb{R}^d) \to H^{-1}(\Omega),
$$
and study its distributional finite element discretization. This operator is part of the distributional de~Rham complex introduced in~\cite{Braess.D;Schoberl.J2008}. The scalar case allows the introduction of the mesh-dependent inner products and norms needed below, and it also reveals the main stability issues in distributional discretizations. This framework is then extended by a tensor-product construction in Section~\ref{subsec:fespaces} to discretize the vector Laplacian, which is the main difficulty for $H(\operatorname{div})$-conforming velocity fields.

The $L^2$ space is discretized by the discontinuous polynomial space
\[
\Sigma_k^{-1}(\mathcal T_h;\mathbb R^d)
= \mathbb{P}_k(\mathcal{T}_h; \mathbb R^d),
\]
which we abbreviate as $\Sigma_k^{-1}$. The superscript $-1$ in $\Sigma_k^{-1}$ indicates that functions are taken elementwise with no
interelement continuity imposed. A natural norm on $\Sigma_k^{-1}$ is the standard $L^2$ norm $\|\cdot\|$.

Next, we introduce the product space of discontinuous polynomial spaces on $\mathcal T_h\times \mathcal F_h$:
\[
M_{k-1,k}^{-1}
:= \mathbb{P}_{k-1}(\mathcal{T}_h) \times \mathbb{P}_k(\mathcal{F}_h).
\]
Again, the superscript $-1$ implies no continuity is imposed between the piecewise polynomial spaces. 
An element of $M_{k-1,k}^{-1}$ is denoted by $u = (u_0, u_b)$.  
A natural product-type inner product is defined by
\[
((u_0, u_b), (v_0, v_b))_{0,h}
= (u_0, v_0) + \sum_{F\in\mathcal F_h} h_F (u_b, v_b)_{F},
\]
where the scaling factor $h_F$ ensures that the two terms are of comparable magnitude.  
The corresponding norm on $M_{k-1,k}^{-1}$ is
\[
\|u\|_{0,h}^2
= \|u_0\|^2 + \sum_{F\in \mathcal F_h} \|h_F^{1/2} u_b\|_{F}^2.
\]
The subspace
\[
\mathring{M}_{k-1,k}^{-1}
= \mathbb{P}_{k-1}(\mathcal{T}_h) \times \mathbb{P}_k(\mathring{\mathcal{F}}_h)
\]
can be viewed as a subspace of ${M}_{k-1,k}^{-1}$ by setting the values on boundary faces to zero.

For $\sigma \in \Sigma_k^{-1}$, define the operator 
$\div_w: \Sigma_k^{-1}\to \mathring{M}_{k-1,k}^{-1}$ by
\begin{equation*}
\div_w \sigma
= \big(\div_T \sigma, - h_F^{-1}[\![\sigma \cdot \boldsymbol{n}]\!]|_F\big)_{T\in \mathcal T_h,~F\in \mathring{\mathcal F}_h}
\in \mathring{M}_{k-1,k}^{-1}.
\end{equation*}
When $\sigma \in \Sigma_k^{-1}\cap H(\div,\Omega)$, we have $\div_w\sigma = (\div\sigma,0)$ since $[\![\sigma \cdot \boldsymbol{n}]\!]|_F = 0$ for all interior faces $F$; in this case we simply write $\div_w\sigma = \div\sigma$.

A graph norm associated with $\div_w$ is defined as
\[
\|\sigma\|_{\div_w}^2
:= \|\sigma\|^2 + \|\div_w\sigma\|_{0,h}^2
= \|\sigma\|^2 + \|\div_h\sigma\|^2
+ \|h^{-1/2}[\![\sigma \cdot \boldsymbol{n}]\!]\|_{\mathring{\mathcal F}_h}^2,
\]
where the scaling is chosen so that the last two terms are of comparable order.  
For $\sigma \in \Sigma_k^{-1}\cap H(\div,\Omega)$, we recover $\|\sigma\|_{\div_w} = \|\sigma\|_{\div}$. 

The following inf-sup condition can be established in a manner similar to Theorem~\ref{thm:infsupcurldiv}, and hence the proof is omitted.

\begin{lemma}\label{lem:divwinfsup}
We have the weak div stability: for a given $k\geq 0$, there exists a constant $\alpha>0$, independent of $h$, but might depend on $k$, such that
\begin{equation}\label{eq:divwinfsup}
 \inf_{v\in \mathring{M}^{-1}_{k-1,k}} 
 \sup_{\tau\in \Sigma_k^{-1}} 
 \frac{(\div_w \tau, v)_{0,h}}{\|\tau\|_{\div_w}\|v\|_{0,h}}
 = \alpha > 0.
\end{equation}
\end{lemma}

Consider the mixed formulation of the Poisson equation $-\Delta u = f$ with Dirichlet boundary condition $u|_{\partial \Omega} = 0$:  
Find $\sigma \in \Sigma_k^{-1}$ and $u\in \mathring{M}_{k-1,k}^{-1}$ such that
\begin{subequations}\label{eq:hy}
 \begin{align}
(\sigma, \tau) + (\div_w \tau, u)_{0,h} &= 0 \qquad\qquad\; \forall\,\tau \in \Sigma_k^{-1},\\
(\div_w \sigma, v)_{0,h} &= -(f, v_0) \quad\; \forall\,v\in  \mathring{M}_{k-1,k}^{-1}.
\end{align}
\end{subequations}
This formulation is precisely the hybridized mixed finite element method using the ${\rm BDM}_k$--$\mathbb P_{k-1}$ pair~\cite{Arnold.D;Brezzi.F1985}.  
The problem \eqref{eq:hy} is well-posed by the inf-sup condition~\eqref{eq:divwinfsup}.


\subsection{Finite element spaces}\label{subsec:fespaces}
To distinguish the vector formulation from the scalar case, we use boldface symbols such as $\boldsymbol{\sigma}$ and $\boldsymbol{u}$. In the Stokes problem, we use the weak divergence operator to discretize the vector Laplacian. Let $\boldsymbol{\sigma}=\nabla \bs u$; then, in the distributional sense,
$-\Delta \bs u = -\div \boldsymbol{\sigma}.$

By the tensor-product construction, we obtain the surjective weak divergence operator
\begin{equation}\label{eq:divM}
\div_w: \Sigma_k^{-1}(\mathbb M) \to \mathring{M}_{k-1,k}^{-1}(\mathbb R^d),
\end{equation}
where
\[
\Sigma_k^{-1}(\mathbb X):=\prod_{T\in \mathcal T_h} \mathbb P_{k}(T; \mathbb X)\quad\text{for }\mathbb X=\mathbb M \text{ or } \mathbb T;\quad \mathring{M}_{k-1,k}^{-1}(\mathbb R^d):=\mathbb R^d\otimes\mathring{M}_{k-1,k}^{-1}.
\]
The tensor space $\Sigma_k^{-1}(\mathbb M)$ in \eqref{eq:divM} will be restricted to the traceless subspace $\Sigma_k^{-1}(\mathbb{T})$, since for $\boldsymbol{\sigma} = \grad \boldsymbol{u}$ with $\div \boldsymbol{u} = 0$, we have $\tr \boldsymbol{\sigma} = \div \boldsymbol{u} = 0$. 
We then impose suitable continuity conditions on the range space and identify the subspaces of $\Sigma_k^{-1}(\mathbb{T})$ that lead to weakly divergence-stable pairs.

\subsubsection*{Spaces for the velocity}
We employ $H(\div)$-conforming finite elements to discretize the velocity field $\boldsymbol{u}$.  
For a simplex $T$ and integers $k \ge 0$ and $\ell = k \text{ or } k - 1$, the local space of shape functions is defined by
\begin{equation*}
\mathbb V_{k,\ell}^{\div}(T)
:= \mathbb P_{k}(T;\mathbb R^d) + \mathbb H_{\ell}(T)\boldsymbol{x}.
\end{equation*}
For $\ell \ge 0$, the degrees of freedom (DoFs) are given by (cf.~\cite[Section~3]{ChenHuang2022})
\begin{subequations}\label{eq:HdivDoF}
\begin{align}
(\boldsymbol v\cdot\boldsymbol n, q)_F, & \quad q\in\mathbb P_{k}(F), \quad F\in\mathcal F(T), \label{Hdivfemdof1}\\
(\boldsymbol v, \boldsymbol q)_T, & \quad \boldsymbol q\in \grad\mathbb P_{\ell}(T)\oplus
\{\boldsymbol q\in \mathbb P_{k-1}(T;\mathbb R^d): \boldsymbol q\cdot\boldsymbol x=0\}. \label{Hdivfemdof2}
\end{align}
\end{subequations}
When $\ell = k$, the space corresponds to the Raviart--Thomas (RT) element~\cite{RaviartThomas1977,Nedelec1980}, where the interior moments in~\eqref{Hdivfemdof2} span $\mathbb{P}_{k-1}(T;\mathbb{R}^d)$. 
For $\ell = k-1$, the Brezzi--Douglas--Marini (BDM) element~\cite{BrezziDouglasMarini1985,Nedelec1986}, the interior moments in~\eqref{Hdivfemdof2} form a strict subspace.

The corresponding global finite element spaces are defined as
\begin{align*}
\mathbb V_{k,\ell}^{\div}(\mathcal T_h)
&:= \{\boldsymbol v_h\in H(\div,\Omega): 
\boldsymbol v_h|_T \in \mathbb V_{k,\ell}^{\div}(T)
\text{ for all } T\in\mathcal T_h\},\\
\mathring{\mathbb V}_{k,\ell}^{\div}(\mathcal T_h)
&:= \mathbb V_{k,\ell}^{\div}(\mathcal T_h)\cap H_0(\div,\Omega).
\end{align*}
For simplicity, we omit $\mathcal T_h$ in the notation afterwards.  

To study the weak div stability, we embed $\mathring{\mathbb V}_{k,\ell}^{\div}$ into $\mathring{M}_{k-1,k}^{-1}(\mathbb R^d)$.

\begin{lemma}
For $k\geq0$, and $\ell=k$ or $k-1$, the embedding map $E: \mathring{\mathbb V}_{k,\ell}^{\div} \to \mathring{M}_{k-1,k}^{-1}(\mathbb R^d)$ defined by
\[
\boldsymbol{u} \mapsto 
\big(Q_{k-1,T}\boldsymbol{u},\, (\boldsymbol{0},\, \boldsymbol{n}_F\cdot \boldsymbol{u})_F\big)_{T\in\mathcal T_h,\, F\in\mathring{\mathcal F}_h}
\]
is injective.
\end{lemma}

\begin{proof}
If $(\boldsymbol{n}_F\cdot \boldsymbol{u})|_F = 0$ and $Q_{k-1,T}\boldsymbol{u} = \bs 0$,  
then all degrees of freedom in~\eqref{eq:HdivDoF} vanish, which implies $\boldsymbol{u}=\bs 0$.  
Hence, $E$ is injective.
\end{proof}

The tangential component $\Pi_F \boldsymbol{u} \in \mathbb{P}_{k}(F; \mathbb{R}^{d-1})$ is not necessarily continuous. We set this component to $\boldsymbol{0}$ here, as the stress space will satisfy tangential--normal continuity.

The domain of $E$ can be extended to $H^1(\Omega;\mathbb R^d)$, but the injectivity no longer holds.

\subsubsection*{Spaces for the  pseudo-stress}
For an integer $k \ge 0$, the local space of shape functions is $\mathbb P_{k}(T;\mathbb T)$,  
with DoFs defined by
\begin{subequations}\label{eq:curldivdof}
\begin{align}\label{eq:curdivfemdof1}
\int_{F} \boldsymbol{t}_i^{\intercal}\boldsymbol{\tau}\boldsymbol{n}\, q \, \mathrm{d}S, 
& \quad q \in \mathbb P_{k}(F), \ i = 1,2,\ldots, d-1,\ F\in\mathcal F(T), \\
\label{eq:curdivfemdof2}
\int_{T} \boldsymbol{\tau}: \boldsymbol{q} \, \mathrm{d}x, 
& \quad \boldsymbol{q} \in \mathbb{P}_{k-1}(T; \mathbb{T}).
\end{align}
\end{subequations}
The local unisolvence of these DoFs for $\mathbb P_{k}(T;\mathbb T)$ can be found in~\cite{ChenHuangZhang2025,GopalakrishnanLedererSchoeberl2020}.  
The corresponding global finite element space is
\begin{align*}
\Sigma_{k}^{\rm tn}(\mathcal T_h)
:= \{\boldsymbol{\tau}_h \in \Sigma_k^{-1}(\mathbb T) :
\text{ all DoFs in~\eqref{eq:curldivdof} are single-valued}\},
\end{align*}
and we abbreviate $\Sigma_{k}^{\rm tn}(\mathcal T_h)$ by $\Sigma_{k}^{\rm tn}$.
Functions in $\Sigma_{k}^{\rm tn}$ are tangential-normal continuous,  
i.e., $[\![\Pi_F\boldsymbol{\sigma}\boldsymbol{n}]\!]_F = 0$ for all $F \in \mathring{\mathcal F}_h$.

The weak divergence operator $\div_{w}$ defines a bilinear form on  
$\Sigma_k^{\rm tn} \times \mathring{\mathbb V}_{k,\ell}^{\div}$:
\begin{equation*}
\begin{aligned}
(\div_{w}\boldsymbol{\sigma}, \boldsymbol{v})_{0,h}
:= (\div_{w}\boldsymbol{\sigma}, E\boldsymbol{v})_{0,h}
&= \sum_{T\in\mathcal{T}_h}(\div\boldsymbol{\sigma}, \boldsymbol{v})_T
- \sum_{F\in\mathring{\mathcal{F}}_h}([\![\boldsymbol{n}^{\intercal}\boldsymbol{\sigma}\boldsymbol{n}]\!], \boldsymbol{n}\cdot\boldsymbol{v})_F\\
&= - \sum_{T\in\mathcal{T}_h}(\boldsymbol{\sigma}, \dev\grad\boldsymbol{v})_T
+ \sum_{F\in\mathcal{F}_h}(\Pi_F\boldsymbol{\sigma}\boldsymbol{n}, [\![\Pi_F\boldsymbol{v}]\!])_F.
\end{aligned}
\end{equation*}

\begin{remark}\rm
The bilinear form $(\div_{w}\cdot, \cdot)_{0,h}$ induces a mapping, still denoted by $\div_w: \Sigma_k^{\rm tn} \to (\mathring{\mathbb V}_{k,\ell}^{\div})'$. 
With the inner product $(\cdot, \cdot)_{0,h}$, we can identify $\mathring{M}_{k-1,k}^{-1}(\mathbb R^d)$ with its dual.  
Through the embedding $E$, we further identify $\mathring{\mathbb V}_{k,\ell}^{\div} \cong E(\mathring{\mathbb V}_{k,\ell}^{\div})$ with its dual,  
and hence can write the mapping as
\begin{equation}\label{eq:weak-div-map}
\div_{w}: \Sigma_k^{\rm tn} \to \mathring{\mathbb V}_{k,\ell}^{\div}.
\end{equation}
\end{remark}

\subsection{Weak div stability}
For each $T \in \mathcal{T}_h$, let $I_T^{\rm tn}: H^1(T;\mathbb{T}) \to \mathbb{P}_{k}(T; \mathbb{T})$  
denote the local interpolation operator defined by DoFs~\eqref{eq:curldivdof}.  
The corresponding global interpolation operator $I_k^{\rm tn}: H^1(\mathcal{T}_h;\mathbb{T}) \to \Sigma_k^{-1}(\mathbb T)$ is given by
\[
(I_k^{\rm tn}\boldsymbol{\tau})|_T := I_T^{\rm tn}(\boldsymbol{\tau}|_T), 
\quad \forall\, T \in \mathcal{T}_h, \ \boldsymbol{\tau} \in H^1(\mathcal{T}_h; \mathbb{T}).
\] 
Then $I_k^{\rm tn}\boldsymbol{\tau}\in \Sigma_{k}^{\rm tn}$ for 
$\boldsymbol{\tau}\in H^1(\Omega;\mathbb T)$.
For any $T\in\mathcal{T}_h$ and $\boldsymbol{\tau}\in H^m(T;\mathbb{T})$ with $1\le m\le k+1$,  
the standard interpolation estimate holds:
\begin{equation}\label{eq:Ihtnprop2}
\|\boldsymbol{\tau}-I_T^{\rm tn}\boldsymbol{\tau}\|_{T}
+ h_T|\boldsymbol{\tau}-I_T^{\rm tn}\boldsymbol{\tau}|_{1,T}
\lesssim h_T^m|\boldsymbol{\tau}|_{m,T}.
\end{equation}

We prove the following commutative property:
\begin{equation}\label{eq:commuting-divw}
Q_{k,k-1}^{\div}\div = \div_w I_k^{\rm tn},
\end{equation}
where $Q_{k,k-1}^{\div}$ is the $L^2$-projection onto $\mathring{\mathbb V}_{k,k-1}^{\div}$, and $\div_w$ is understood as in \eqref{eq:weak-div-map}.

\begin{lemma}\label{lm:interpolant}
For $\bs \tau \in H^1(\Omega;\mathbb T)$, we have 
\begin{equation}\label{eq:divinterpolant}
(\div \bs \tau, \bs v_h) = (\div_w I_k^{\rm tn} \bs \tau, \bs v_h)_{0,h}, \quad \forall\,\bs v_h\in \mathring{\mathbb V}_{k,k-1}^{\div}.
\end{equation}
\end{lemma}
\begin{proof}
Notice that $\dev\grad_h \bs v_h\in \Sigma_{k-1}^{-1}(\mathbb T)$ and $[\![\Pi_F\bs v_h]\!]|_F\in \mathbb P_{k}(F;\mathbb R^{d-1})$ for $F\in\mathcal{F}_h$. So
\begin{align*}
(\div_w I_k^{\rm tn} \bs \tau, \bs v_h)_{0,h} &= - \sum_{T\in\mathcal{T}_h}(I_k^{\rm tn} \bs \tau, \dev\grad\boldsymbol{v}_h)_T + \sum_{F\in\mathcal{F}_h}(\Pi_F(I_k^{\rm tn} \bs \tau) \boldsymbol{n}, [\![\Pi_F\boldsymbol{v}_h]\!])_F\\
&= - \sum_{T\in\mathcal{T}_h}(\bs \tau, \dev\grad\boldsymbol{v}_h)_T + \sum_{F\in\mathcal{F}_h}(\Pi_F\bs \tau \boldsymbol{n}, [\![\Pi_F\boldsymbol{v}_h]\!])_F\\
& = (\div \bs \tau, \bs v_h).
\end{align*}
\end{proof}
This leads to the following inf-sup condition.

\begin{theorem}
There exists a constant $\alpha>0$, independent of $h$, such that
\begin{equation}\label{eq:infsupcurldivw}
\inf_{\boldsymbol{v}_h\in \mathring{\mathbb V}_{k,k-1}^{\div}} 
\sup_{\boldsymbol{\tau}_h\in \Sigma_{k}^{\rm tn}} 
\frac{(\div_w \boldsymbol{\tau}_h, \boldsymbol{v}_h)_{0,h}}
{\|\boldsymbol{\tau}_h\|_{\div_w}\,\|\boldsymbol{v}_h\|} 
= \alpha.
\end{equation} 
\end{theorem}


\begin{proof}
		For any $\boldsymbol{v}_h\in \mathring{\mathbb V}_{k,k-1}^{\div}\subset L^2(\Omega;\mathbb R^d)$,  
		the surjectivity of the divergence operator mapping from $H^1(\Omega;\mathbb T)$ onto $L^2(\Omega;\mathbb R^d)$  (cf. \cite[(33)]{ArnoldHu2021}) guarantees that there exists $\boldsymbol{\tau}\in H^1(\Omega;\mathbb T)$ such that  
		$\div \boldsymbol{\tau} = \boldsymbol{v}_h$ and $\|\boldsymbol{\tau}\|_1 \lesssim \|\boldsymbol{v}_h\|$.  
		Let $\boldsymbol{\tau}_h = I_k^{\rm tn}\boldsymbol{\tau}$; then the commutative property~\eqref{eq:divinterpolant}
		implies $\div_w \boldsymbol{\tau}_h = \boldsymbol{v}_h$, which, combined with the boundedness of the interpolant $I_k^{\rm tn}$, yields the desired result~\eqref{eq:infsupcurldivw}.
\end{proof}

\begin{remark}\rm
If the range space is enriched to ${\rm RT}_k$, i.e.,  
$\boldsymbol{v}_h \in \mathring{\mathbb V}_{k,k}^{\div}$, then  
$\dev\grad_h \boldsymbol{v}_h \in \Sigma_k^{-1}(\mathbb T)$ not $\Sigma_{k-1}^{-1}(\mathbb T)$. 
Since DoF~\eqref{eq:curdivfemdof2} to define $I_k^{\rm tn}$ is for moments of degree $k-1$, the commutative property~\eqref{eq:divinterpolant} no longer holds.  
Consequently, the inf-sup condition \eqref{eq:infsupcurldivw} may fail for the pair  
$\Sigma_k^{\rm tn}$-$\mathring{\mathbb V}_{k,k}^{\div}$.
\end{remark}

\begin{remark}\label{rm:infsupsigma0}
On the other hand, to accommodate the slip boundary condition, we can reduce the stress space to impose the boundary condition $\Pi_F\boldsymbol{\sigma}\boldsymbol{n}=\boldsymbol{0}$ on $\partial\Omega$. Let 
\[
\mathring{H}^1(\Omega;\mathbb T) := \{\boldsymbol{\tau}\in H^1(\Omega;\mathbb T): \Pi_F\boldsymbol{\tau}\boldsymbol{n} = \boldsymbol{0} \text{ on } \partial\Omega\}.
\]
Following the proof of \eqref{eq:infsupcurldivw}, and using the surjectivity
$
\div \mathring{H}^1(\Omega;\mathbb T)=L^2(\Omega;\mathbb R^d)
$
together with the commuting property~\eqref{eq:divinterpolant}, one can show that
\begin{equation*}
\inf_{\boldsymbol{v}_h\in \mathring{\mathbb V}_{k,k-1}^{\div}} 
\sup_{\boldsymbol{\tau}_h\in \mathring{\Sigma}_{k}^{\rm tn}} 
\frac{(\div_w \boldsymbol{\tau}_h, \boldsymbol{v}_h)_{0,h}}
{\|\boldsymbol{\tau}_h\|_{\div_w}\,\|\boldsymbol{v}_h\|} 
= \alpha
\end{equation*} 
for some constant $\alpha>0$ independent of $h$, where
\[
\mathring{\Sigma}_{k}^{\rm tn}:= \{\boldsymbol{\tau}_h \in \Sigma_k^{\rm tn} : \Pi_F\boldsymbol{\tau}_h\boldsymbol{n} = \boldsymbol{0} \text{ on } \partial\Omega\}.
\] 
\end{remark}

Define the weak deviatoric gradient 
$\dev\grad_w: \mathring{\mathbb V}_{k,\ell}^{\div}\to \Sigma_k^{\rm tn}$ 
as the adjoint of $\div_w$.  
For $\boldsymbol{v}\in \mathring{\mathbb V}_{k,\ell}^{\div}$, 
$\dev\grad_w \boldsymbol{v}\in \Sigma_k^{\rm tn}$ is defined by
\[
(\dev\grad_w \boldsymbol{v}, \boldsymbol{\tau}) 
= -(\boldsymbol{v}, \div_w \boldsymbol{\tau})_{0,h}, 
\quad \forall\, \boldsymbol{\tau}\in \Sigma_k^{\rm tn}.
\]
Since $\div_{w}\Sigma_k^{\rm tn}=\mathring{\mathbb V}_{k,k-1}^{\div}$,  
the operator 
$\dev\grad_w: \mathring{\mathbb V}_{k,k-1}^{\div}\to \Sigma_k^{\rm tn}$ 
is injective, and thus $\|\dev\grad_w(\cdot)\|$ defines a norm on  
$\mathring{\mathbb V}_{k,k-1}^{\div}$.
Consequently, the inf-sup condition~\eqref{eq:infsupcurldivw} can be written equivalently as
\begin{equation}\label{eq:infsupdevgrad}
 \inf_{\boldsymbol{v}_h\in \mathring{\mathbb V}_{k,k-1}^{\div}} 
 \sup_{\boldsymbol{\tau}_h\in \Sigma_{k}^{\rm tn}} 
 \frac{(\div_w \boldsymbol{\tau}_h, \boldsymbol{v}_h)_{0,h}}
 {\|\boldsymbol{\tau}_h\|\,\|\dev\grad_w \boldsymbol{v}_h\|} 
 = 1 > 0.
\end{equation}


With these preparations, we consider the mixed finite element method for the vector Laplacian:
Find $\boldsymbol{u} \in H_0^1(\Omega; \mathbb{R}^d)$ such that $-\Delta \boldsymbol{u} = \boldsymbol{f}$. The discrete problem seeks $\boldsymbol{\sigma}_h \in \Sigma_k^{\rm tn}$ and $\boldsymbol{u}_h \in \mathring{\mathbb{V}}_{k,k-1}^{\div}$ satisfying:
\begin{equation}\label{eq:vecLapBDM}
\begin{cases}
(\boldsymbol{\sigma}_h, \boldsymbol{\tau}_h) + (\boldsymbol{u}_h, \div_w \boldsymbol{\tau}_h)_{0,h} = 0 & \forall\,\boldsymbol{\tau}_h \in \Sigma_k^{\rm tn}, \\
(\div_w\boldsymbol{\sigma}_h, \boldsymbol{v}_h)_{0,h} = -(\boldsymbol{f}, \boldsymbol{v}_h) & \forall\,\boldsymbol{v}_h \in \mathring{\mathbb{V}}_{k,k-1}^{\div}.
\end{cases}
\end{equation}
The mixed method \eqref{eq:vecLapBDM} is well-posed due to the inf-sup conditions \eqref{eq:infsupcurldivw} and \eqref{eq:infsupdevgrad}. However, a standard error analysis reveals that $\|\boldsymbol{\sigma} - \boldsymbol{\sigma}_h\|$ and $\|\dev\grad \boldsymbol{u} - \dev\grad_w I^{\div}_{k,k-1}\boldsymbol{u}\|$ are coupled. Subtracting the continuous system from \eqref{eq:vecLapBDM} yields:
$$
\begin{aligned}
(\boldsymbol{\sigma}-\boldsymbol{\sigma}_h,\boldsymbol{\tau}_h) + (\boldsymbol{u}-\boldsymbol{u}_h,\div_w\boldsymbol{\tau}_h)_{0,h} &= 0, \quad \forall\,\boldsymbol{\tau}_h \in \Sigma_k^{\rm tn}, \\
(\div_w(\boldsymbol{\sigma}-\boldsymbol{\sigma}_h),\boldsymbol{v}_h)_{0,h} &= 0, \quad \forall\,\boldsymbol{v}_h \in \mathring{\mathbb{V}}_{k,k-1}^{\div}.
\end{aligned}
$$
The term $(\boldsymbol{u}-\boldsymbol{u}_h,\div_w\boldsymbol{\tau}_h)_{0,h} = -(\dev\grad \boldsymbol{u} - \dev\grad_w \boldsymbol{u}_h, \bs \tau_h)$ is bounded by the $H^1$-type error $\|\dev\grad \boldsymbol{u} - \dev\grad_w \boldsymbol{u}_h\|$ rather than by $\|\boldsymbol{\sigma}-\boldsymbol{\sigma}_h\|$ alone. This suboptimality arises because being weakly divergence-free with respect to $\mathring{\mathbb{V}}_{k,k-1}^{\div}$ does not imply strongly divergence-free properties:
$$ (\div_w\boldsymbol{\sigma}_h, \boldsymbol{v}_h)_{0,h} = 0 \quad \forall\,\boldsymbol{v}_h \in \mathring{\mathbb{V}}_{k,k-1}^{\div} \quad \nRightarrow \quad \langle \div\boldsymbol{\sigma}_h, \boldsymbol{v}\rangle = 0 \quad \forall\,\boldsymbol{v} \in H^1(\Omega;\mathbb{R}^d). $$

However, we will demonstrate in Lemma \ref{lm:divfree} that this property is preserved if the velocity test space in \eqref{eq:vecLapBDM} is enriched from $\mathring{\mathbb{V}}_{k,k-1}^{\div}$ to $\mathring{\mathbb{V}}_{k,k}^{\div}$. In this enriched setting, the test space is sufficiently large to ensure that for $\boldsymbol{\tau}_h = I_h^{\rm tn}\boldsymbol{\sigma} - \boldsymbol{\sigma}_h$, $\div_w \boldsymbol{\tau}_h = 0$ holds. Consequently, the coupling term $(\boldsymbol{u}-\boldsymbol{u}_h, \div_w \boldsymbol{\tau}_h)_{0,h} = 0$ vanishes, yielding the optimal stress estimate.
By decoupling the stress error from the $O(h^k)$ velocity approximation error, we avoid the lower-order contamination and achieve the full $O(h^{k+1})$ convergence rate.

\subsection{Commuting property for the weak dev grad operator}
For each $T \in \mathcal{T}_h$, let $I_{(k,\ell),T}^{\rm div}: H^1(T;\mathbb{R}^d) \to \mathbb V_{k,\ell}^{\div}(T)$  
denote the local interpolation operator defined by the DoFs in~\eqref{eq:HdivDoF}.  
The corresponding global interpolation operator $I_{k,\ell}^{\div}: H^1(\mathcal{T}_h;\mathbb{R}^d)\to L^2(\Omega;\mathbb R^d)$ is given by
\[
(I_{k,\ell}^{\div}\boldsymbol{v})|_T := I_{(k,\ell),T}^{\rm div}(\boldsymbol{v}|_T), 
\quad \forall\, T \in \mathcal{T}_h, \ \boldsymbol{v} \in H^1(\mathcal{T}_h; \mathbb{R}^d).
\] 
Then $I_{k,\ell}^{\div}\boldsymbol{v}\in \mathbb V_{k,\ell}^{\div}$ for any 
$\boldsymbol{v}\in H^1(\Omega;\mathbb R^d)$.  
For $T\in\mathcal{T}_h$ and $\boldsymbol{v}\in H^m(T;\mathbb{R}^d)$ with $1\le m\le k+1$,  
the interpolation satisfies
\begin{equation}\label{eq:Ihdivprop2}
\|\boldsymbol{v}-I_{(k,\ell),T}^{\rm div}\boldsymbol{v}\|_{T}
+ h_T|\boldsymbol{v}-I_{(k,\ell),T}^{\rm div}\boldsymbol{v}|_{1,T}
\lesssim h_T^m|\boldsymbol{v}|_{m,T}.
\end{equation}

We are going to prove the following commutative property:
\begin{equation}\label{eq:Qkgrad}
Q_{k}^{\rm tn}\dev \grad = \dev\grad_w I_{k,k}^{\div},
\end{equation}
where $Q_k^{\rm tn}: L^2(\Omega; \mathbb T)\to \Sigma_{k}^{\rm tn}$ is the $L^2$-projection.

\begin{lemma}\label{lm:devinterpolant}
For any $\boldsymbol{u}\in H_0^1(\Omega; \mathbb R^d)$,  
\begin{equation}\label{eq:devinterpolant}
(\dev \grad \boldsymbol{u}, \boldsymbol{\tau}_h) 
=  (\dev\grad_w I_{k,k}^{\div} \boldsymbol{u}, \boldsymbol{\tau}_h), 
\quad \forall\,\boldsymbol{\tau}_h\in \Sigma_k^{\rm tn}.
\end{equation}
\end{lemma}
\begin{proof}
As $\div(\bs \tau_h|_T)\in \mathbb P_{k-1}(T;\mathbb R^d)$ and $[\![\bs n^{\intercal}\bs \tau_h\bs n]\!]|_F\in \mathbb P_k(F)$, we have
\begin{align*}
(\dev \grad \bs u, \bs \tau_h) &= -\sum_{T\in\mathcal{T}_h} (\bs u, \div \bs \tau_h)_T + \sum_{F\in\mathring{\mathcal{F}}_h}(\bs u\cdot \bs n, [\![\bs n^{\intercal}\bs \tau_h\bs n]\!])_F \\
 &= -\sum_{T\in\mathcal{T}_h} (I_{k,k}^{\div} \bs u, \div \bs \tau_h)_T + \sum_{F\in\mathring{\mathcal{F}}_h}(I_{k,k}^{\div} \bs u\cdot \bs n, [\![\bs n^{\intercal}\bs \tau_h\bs n]\!])_F \\
 &= (\dev\grad_w I_{k,k}^{\div} \bs u, \bs \tau_h).
\end{align*}
\end{proof}

The $\mathrm{RT}_k$ space is required to ensure the interior moments~\eqref{Hdivfemdof2} cover
$\mathbb P_{k-1}(T;\mathbb R^d)$ for $\div \boldsymbol{\tau}_h$, 
while for the $\mathrm{BDM}_k$ space $\mathbb V_{k,k-1}^{\div}$,  
the interior moments are not sufficient.  
Hence, the commutative property~\eqref{eq:devinterpolant} does not hold for  
$I_{k,k-1}^{\div}$.

The weak operator $\div_w$ can be regarded as a discretization of $\div$ in the distributional sense.  
We can preserve the strongly divergence-free property in the following sense.

\begin{lemma}\label{lm:divfree}
Let $\boldsymbol{\sigma}_h\in \Sigma_k^{\rm tn}$.  
If
\[
(\div_w\boldsymbol{\sigma}_h, \boldsymbol{v}_h)_{0,h} = 0, 
\quad \forall\, \boldsymbol{v}_h\in \mathring{\mathbb V}_{k,k}^{\div},
\]
then $\div\boldsymbol{\sigma}_h = 0$ in the sense of distributions.
\end{lemma}
\begin{proof}
It suffices to prove $\div\boldsymbol{\sigma}_h = 0$ in the distributional sense:
\[
\langle \div \boldsymbol{\sigma}_h, \boldsymbol{v}\rangle = 0, 
\quad \forall\, \boldsymbol{v}\in C_0^{\infty}(\Omega;\mathbb R^d).
\]
Let $\bs v_I = I_{k,k}^{\div}\bs v$. Then by Lemma \ref{lm:devinterpolant},
\begin{align*}
\langle \div \bs \sigma_h, \bs v\rangle = - (\bs \sigma_h, \dev\grad \bs v)= -(\bs \sigma_h, \dev\grad_w \bs v_I ) = (\div_w\bs \sigma_h, \bs v_I)_{0,h} = 0.
\end{align*}
\end{proof}

\subsection{Restriction to divergence-free subspaces}
The commuting conditions \eqref{eq:divinterpolant} and \eqref{eq:devinterpolant} involve different velocity spaces, as shown in diagram \eqref{diag:main}:
\begin{equation}\label{diag:main}
\begin{tikzcd}[column sep=large, row sep=large]
H^1(\Omega; \mathbb{T}) \arrow[r, "\div"] \arrow[d, "I_k^{\rm tn}"'] & L^2(\Omega; \mathbb{R}^d) \arrow[d, "Q_{k,k-1}^{\div}"] \\
\Sigma_k^{\rm tn} \arrow[r, "\div_w"] & \mathring{\mathbb{V}}_{k,k-1}^{\div}
\end{tikzcd}
\quad
\begin{tikzcd}[column sep=large, row sep=large]
H_0^1(\Omega; \mathbb{R}^d) \arrow[r, "\dev \grad"] \arrow[d, "I_{k,k}^{\div}"'] & L^2(\Omega; \mathbb{T}) \arrow[d, "Q_{k}^{\rm tn}"] \\
\mathring{\mathbb{V}}_{k,k}^{\div} \arrow[r, "\dev \grad_w"] & \Sigma_k^{\rm tn}
\end{tikzcd}
\end{equation}
Weak divergence stability holds on the smaller space $\mathring{\mathbb{V}}_{k,k-1}^{\div}$, whereas strongly divergence-freeness requires the larger space $\mathring{\mathbb{V}}_{k,k}^{\div}$. We resolve this structural mismatch by restricting our analysis to the divergence-free subspace. 

Recall the standard property for $H(\div)$ interpolants \cite{ArnoldFalkWinther2006}:
\begin{equation}\label{eq:Ihdivprop1}
\div(I_{k,\ell}^{\div}\boldsymbol{v}) = Q_{\ell}(\div\boldsymbol{v}), \quad \forall\, \boldsymbol{v}\in H^1(\Omega;\mathbb{R}^d).
\end{equation}
Equation \eqref{eq:Ihdivprop1} implies that $I_{k,\ell}^{\div}\boldsymbol{v} \in \ker(\div)$ whenever $\boldsymbol{v} \in \ker(\div)$. Notably, for $k \ge 1$, the additional interior moments in $\mathring{\mathbb{V}}_{k,k}^{\div}$ compared to $\mathring{\mathbb{V}}_{k,k-1}^{\div}$ only enrich the range of the divergence operator and do not alter the divergence-free part of the space. Consequently, for any $\boldsymbol{v} \in H_0^1(\Omega;\mathbb{R}^d) \cap \ker(\div)$, we have:
\begin{equation}\label{eq:Ikk}
I_{k,k}^{\div}\boldsymbol{v} = I_{k,k-1}^{\div}\boldsymbol{v} \in \mathring{\mathbb{V}}_{k,k-1}^{\div} \cap \ker(\div).
\end{equation}
This identity bridges the two diagrams in \eqref{diag:main}, allowing us to leverage the stability of $\mathring{\mathbb{V}}_{k,k-1}^{\div}$ alongside the consistency of the weak deviatoric gradient on $\mathring{\mathbb{V}}_{k,k}^{\div}$.

Consider the mixed finite element method for the vector Laplacian
\begin{equation}\label{eq:vectorLapdivfree}
-\Delta \boldsymbol{u} = \boldsymbol{f}, \quad \text{with } \boldsymbol{u}\in H_0^1(\Omega; \mathbb R^d)\cap \ker(\div).
\end{equation}
Find $\boldsymbol{\sigma}_h\in \Sigma_k^{\rm tn}$ and $\boldsymbol{u}_h\in \mathring{\mathbb V}_{k,\ell}^{\div}\cap \ker(\div)$ such that
\begin{subequations}\label{eq:vecLapdivfree}
\begin{align}
\label{eq:vecLapdivfree1}
(\boldsymbol{\sigma}_h, \boldsymbol{\tau}_h) + (\boldsymbol{u}_h, \div_w \boldsymbol{\tau}_h)_{0,h} &= 0, \qquad\qquad\; \forall\, \boldsymbol{\tau}_h\in \Sigma_k^{\rm tn},\\
\label{eq:vecLapdivfree2}
(\div_w \boldsymbol{\sigma}_h, \boldsymbol{v}_h)_{0,h} &= -(\boldsymbol{f}, \boldsymbol{v}_h), \quad \forall\, \boldsymbol{v}_h\in \mathring{\mathbb V}_{k,\ell}^{\div}\cap \ker(\div).
\end{align}
\end{subequations}

By the inf-sup condition \eqref{eq:infsupcurldivw} and the commuting property \eqref{eq:devinterpolant}, we obtain the following two properties.
\begin{lemma}
There exists a constant $\alpha>0$, independent of $h$, such that
\begin{equation}\label{eq:infsupdivfree}
\inf_{\boldsymbol{v}_h\in \mathring{\mathbb V}_{k,\ell}^{\div}\cap \ker(\div)} 
\sup_{\boldsymbol{\tau}_h\in \Sigma_{k}^{\rm tn}} 
\frac{(\div_w \boldsymbol{\tau}_h, \boldsymbol{v}_h)_{0,h}}
{\|\boldsymbol{\tau}_h\|_{\div_w}\,\|\boldsymbol{v}_h\|} 
= \alpha.
\end{equation} 
\end{lemma}
\begin{proof}
It follows from the inf-sup condition \eqref{eq:infsupcurldivw} and the fact  \[\mathring{\mathbb V}_{k,k}^{\div}\cap \ker(\div) = \mathring{\mathbb V}_{k,k-1}^{\div}\cap \ker(\div).
\] 
\end{proof}
\begin{lemma}\label{lm:strongdivfree}
We have the distributional divergence-free property against divergence-free velocity fields: if
\[
(\operatorname{div}_w \boldsymbol{\tau}_h, \boldsymbol{v}_h)_{0,h} = 0
\qquad \forall\, \boldsymbol{v}_h \in \mathring{\mathbb V}_{k,\ell}^{\div}\cap \ker(\div),
\]
then
\[
\langle \operatorname{div} \boldsymbol{\tau}_h, \boldsymbol{v}\rangle = 0
\qquad \forall\, \boldsymbol{v} \in H_0^1(\Omega;\mathbb R^d)\cap \ker(\operatorname{div}).
\]
\end{lemma}
\begin{proof}
By the definition of distributional $\div$, the commuting property \eqref{eq:devinterpolant}, and \eqref{eq:Ikk}
$$
\begin{aligned}
\langle \operatorname{div} \boldsymbol{\tau}_h, \boldsymbol{v}\rangle = - ( \boldsymbol{\tau}_h, \dev\grad \boldsymbol{v}) =  - (\boldsymbol{\tau}_h, \dev\grad_w I_{k,k}^{\div} \boldsymbol{v}) = (\div_w\boldsymbol{\tau}_h, I_{k,k}^{\div} \boldsymbol{v})_{0,h} = 0.  
\end{aligned}
$$ 
\end{proof}

\begin{theorem}\label{thm:divfreemixfem}
Let $k\geq0$, and $\ell\in\{k,k-1\}$. The mixed finite element method \eqref{eq:vecLapdivfree} is well-posed. Let $\boldsymbol{\sigma}=\dev\grad \boldsymbol{u}$, where $\boldsymbol{u}$ solves \eqref{eq:vectorLapdivfree}. Then
\begin{equation}\label{eq:sigmaestimate}
\|\dev\grad \boldsymbol{u} - \dev\grad_w \boldsymbol{u}_h\|
= \|\boldsymbol{\sigma} - \boldsymbol{\sigma}_h\|
\le \|\boldsymbol{\sigma} - I_k^{\rm tn}\boldsymbol{\sigma}\|,
\end{equation}
and 
\begin{equation}\label{eq:uIuhestimate}
\|\dev\grad_w  (I_{k,k}^{\div}\boldsymbol{u} - \boldsymbol{u}_h)\|
= \| Q_k^{\rm tn}(\boldsymbol{\sigma} - \boldsymbol{\sigma}_h)\|
\le \|\boldsymbol{\sigma} - I_k^{\rm tn}\boldsymbol{\sigma}\|.
\end{equation}
Moreover, if $\boldsymbol{u}\in H^{k+2}(\Omega;\mathbb R^d)$, then
\begin{equation}\label{eq:uestimate}
\|\dev\grad_w  (I_{k,k}^{\div}\boldsymbol{u} - \boldsymbol{u}_h)\| + \|\dev\grad \boldsymbol{u} - \dev\grad_w \boldsymbol{u}_h\|
\lesssim h^{k+1}\|\boldsymbol{u}\|_{k+2}.
\end{equation}
\end{theorem}
\begin{proof}
The well-posedness is from \eqref{eq:infsupdivfree}. 
Denote by $\bs \sigma_I =  I_k^{\rm tn}\bs \sigma$ and $\bs u_I = I_{k,k}^{\div}\boldsymbol{u} \in \mathring{\mathbb V}_{k,k-1}^{\div}\cap \ker(\div)$. By \eqref{eq:divinterpolant}, it holds that $(\div_w \bs \sigma_I, \bs v_h)_{0,h} = (\div \bs \sigma, \bs v_h)= -(\bs f, \bs v_h) = (\div_w \bs \sigma_h, \bs v_h)_{0,h}$, so we conclude
\[
(\div_w (\bs \sigma_I - \bs \sigma_h), \bs v_h)_{0,h} = 0, \quad \forall\, \bs v_h\in \mathring{\mathbb V}_{k,\ell}^{\div}\cap \ker(\div).
\]
So by Lemma \ref{lm:strongdivfree}, as $\div \bs u = 0$, 
$$(\dev\grad\bs u, \bs \sigma_I - \bs \sigma_h) = - \langle \div(\bs \sigma_I - \bs \sigma_h), \bs u \rangle = 0.$$ 
We obtain the following orthogonality
\begin{align*}
(\bs \sigma - \bs \sigma_h, \bs \sigma_I - \bs \sigma_h) &= (\dev\grad\bs u, \bs \sigma_I - \bs \sigma_h) + (\div_w (\bs \sigma_I - \bs \sigma_h), \bs u_h)_{0,h} = 0.
\end{align*}
The estimate \eqref{eq:sigmaestimate} follows from this orthogonality:
\[
\|\bs \sigma - \bs \sigma_h\|^2 = (\bs \sigma - \bs \sigma_h, \bs \sigma - \bs \sigma_I)\leq \|\bs \sigma - \bs \sigma_h\|\|\bs \sigma - \bs \sigma_I\|.
\]

By the commutative property \eqref{eq:Qkgrad}, $\dev\grad_w\bs u_I = Q_k^{\rm tn} \dev\grad \bs u = Q_k^{\rm tn}\bs \sigma$. Then \eqref{eq:uIuhestimate} follows.
\end{proof}

Since $\mathbb{P}_k(T; \mathbb{R}^d) \subseteq \mathbb{V}_{k,\ell}^{\div}(T)$, the standard approximation theory for the velocity space $\mathring{\mathbb{V}}_{k,\ell}^{\div}$ yields an $H^1$-type error of order $O(h^k)$. Consequently, the estimate \eqref{eq:uestimate}, which establishes a higher convergence rate $O(h^{k+1})$, constitutes a velocity superconvergence result.

\section{Mixed Finite Element Methods for the Stokes Equation}\label{sec:stokesmfem}
In this section, we develop a mixed finite element discretization for the Stokes equation that ensures the pointwise divergence-free velocity. We introduce Lagrange multiplier spaces using piecewise polynomials for pressure and facial multipliers to relax the tangential-normal continuity of the pseudo-stress. After proving stability via a discrete inf-sup condition, we formulate the mixed scheme and show its equivalence to \eqref{eq:vecLapdivfree}. The resulting error analysis yields superconvergent and pressure-robust bounds. Finally, we establish $L^2$-velocity superconvergence of order $O(h^{k+2})$ and present a postprocessing strategy that achieves $H^1$-superconvergence.


\subsection{Spaces of Lagrange multipliers}

The subspace $\mathring{\mathbb V}_{k,\ell}^{\div}\cap \ker(\div)$ is not constructed explicitly.  
Instead, we enforce the constraint $\div \boldsymbol{u} = 0$ by introducing a Lagrange multiplier, which has the physical interpretation of pressure.

In the discretization, we use the piecewise polynomial space $\mathbb{P}_{\ell}(\mathcal{T}_h)$ to approximate the pressure $p$.  
By the commutative property \eqref{eq:Ihdivprop1}, it follows that
\begin{equation*}
\div\mathring{\mathbb{V}}_{k,\ell}^{\div}=\mathbb{P}_{\ell}(\mathcal{T}_h)/\mathbb R.
\end{equation*}

To relax the tangential-normal continuity condition in the space $\Sigma_k^{\rm tn}$, we introduce the Lagrange multiplier spaces  
\begin{equation*}
\Lambda_k = \mathbb P_k(\mathcal F_h; \mathbb R^{d-1}) \quad\textrm{and}\quad  
\mathring{\Lambda}_k = \mathbb P_k(\mathring{\mathcal F}_h; \mathbb R^{d-1}). 
\end{equation*}
The embedding $E: \mathring{\mathbb V}_{k,\ell}^{\div} \times \mathring{\Lambda}_k \to \mathring{M}_{k-1,k}^{-1}(\mathbb R^d)$  
is defined as
\[
(\boldsymbol{u}, \boldsymbol{\lambda}) \mapsto 
\big(Q_{k-1,T}\boldsymbol{u},\, (\boldsymbol{\lambda},\, \boldsymbol{n}_F\cdot \boldsymbol{u})_F\big)_{T\in\mathcal T_h,\, F\in\mathring{\mathcal F}_h},
\]
which remains injective.

The stress space is relaxed to  
$\Sigma_k^{-1}(\mathbb T) = \mathbb P_k(\mathcal T_h; \mathbb T)$.  
The weak divergence operator $\div_w$ then induces the bilinear form on $\Sigma_k^{-1}(\mathbb T)\times (\mathring{\mathbb{V}}_{k,\ell}^{\div}\times \mathring{\Lambda}_k)$:
\begin{equation}\label{eq:divwT}
(\div_w\boldsymbol{\sigma}, (\boldsymbol{v}, \boldsymbol{\mu}))_{0,h}
=\! \sum_{T\in \mathcal T_h}(\div \boldsymbol{\sigma}, \boldsymbol{v})_T  
- \!\sum_{F\in\mathring{\mathcal{F}}_h}\!\Big(
([\![\boldsymbol{n}^{\intercal}\boldsymbol{\sigma}\boldsymbol{n}]\!],\, \boldsymbol{n}\cdot\boldsymbol{v})_F 
+ ([\![\Pi_F\boldsymbol{\sigma}\boldsymbol{n}]\!],\, \boldsymbol{\mu})_F\Big),
\end{equation}
which induces a mapping, still denoted by $\div_w: \Sigma_k^{-1}(\mathbb T)\to (\mathring{\mathbb{V}}_{k,\ell}^{\div}\times \mathring{\Lambda}_k)'\cong (\mathring{\mathbb{V}}_{k,\ell}^{\div}\times \mathring{\Lambda}_k)$ by the $(\cdot,\cdot)_{0,h}$ inner product through the embedding $E$. 

Its adjoint, still denoted by  
$\dev\grad_w: \mathring{\mathbb V}_{k,\ell}^{\div} \times \mathring{\Lambda}_k \to \Sigma_k^{-1}(\mathbb T)$,  
is defined by
\begin{equation*}
(\boldsymbol{\sigma}, \dev\grad_w (\boldsymbol{v}, \boldsymbol{\mu})) := -(\div_w\boldsymbol{\sigma}, (\boldsymbol{v}, \boldsymbol{\mu}))_{0,h},
\end{equation*}
for $\boldsymbol{\sigma}\in \Sigma_k^{-1}(\mathbb T)$,  
$\boldsymbol{v}\in \mathring{\mathbb V}_{k,\ell}^{\div}$, and  
$\boldsymbol{\mu}\in \mathring{\Lambda}_k$.

The operator $\dev \grad_w$ can be extended to  
\[
\dev \grad_w: H^1(\mathcal{T}_h;\mathbb{R}^d)\times L^2(\mathcal{F}_h;\mathbb R^{d-1}) 
\to \Sigma_k^{-1}(\mathbb T)
\] 
as follows: for $(\boldsymbol{v}, \boldsymbol{\mu})\in H^1(\mathcal{T}_h;\mathbb{R}^d)\times L^2(\mathcal{F}_h;\mathbb R^{d-1})$,  
let $\dev \grad_w(\boldsymbol{v}, \boldsymbol{\mu})\in \Sigma_k^{-1}(\mathbb T)$ be defined elementwise by
\begin{equation*}
(\dev \grad_w(\boldsymbol{v}, \boldsymbol{\mu}), \boldsymbol{\tau})_T
= -(\boldsymbol{v},\div\boldsymbol{\tau})_T
+ (\boldsymbol{n}\cdot\boldsymbol{v},\, \boldsymbol{n}^{\intercal}\boldsymbol{\tau}\boldsymbol{n})_{\partial T}
+ (\boldsymbol{\mu},\, \Pi_F\boldsymbol{\tau}\boldsymbol{n})_{\partial T},
\end{equation*}
for all $\boldsymbol{\tau}\in \mathbb{P}_{k}(T;\mathbb{T})$ and $T\in\mathcal{T}_h$.

\begin{lemma}\label{lm:weakdevgrad}
For $\boldsymbol{v}\in H^1(\Omega;\mathbb R^d)$ and $\boldsymbol{\mu}\in L^2(\mathcal{F}_h;\mathbb R^{d-1})$,  
it holds that
\begin{equation}\label{eq:weakdevgradcommu}
\dev \grad_w(I_{k,k}^{\div}\boldsymbol{v},\, Q_{k,\mathcal{F}_h}\boldsymbol{\mu}) 
= \dev \grad_w(\boldsymbol{v},\, \boldsymbol{\mu}),
\end{equation}
where $Q_{k,\mathcal{F}_h}$ is the $L^2$-orthogonal projector from $L^2(\mathcal{F}_h;\mathbb R^{d-1})$ onto $\Lambda_k$.
\end{lemma}
\begin{proof}
By the definition of the operator $\dev \grad_w$,	
we have for any $\boldsymbol{\tau}\in \mathbb{P}_{k}(T;\mathbb{T})$ and $T\in\mathcal{T}_h$ that
\begin{align*}
&\quad\; (\dev \grad_w(I_{k,k}^{\div}\boldsymbol{v}, Q_{k,\mathcal{F}_h}\boldsymbol{\mu}), \boldsymbol{\tau})_T \\
&=-(I_{k,k}^{\div}\boldsymbol{v},\div\boldsymbol{\tau})_T+(\boldsymbol{n}\cdot(I_{k,k}^{\div}\boldsymbol{v}),\boldsymbol{n}^{\intercal}\boldsymbol{\tau}\boldsymbol{n})_{\partial T}+(\boldsymbol{\mu}, \Pi_F\boldsymbol{\tau}\boldsymbol{n})_{\partial T}.
\end{align*}
Then \eqref{eq:weakdevgradcommu} follows from the definition of the interpolation operator $I_{k,k}^{\div}$.
\end{proof}

Again, \eqref{eq:weakdevgradcommu} may not hold for $I_{k,k-1}^{\div}$, because to match the term $(I_{k,k}^{\div}\boldsymbol{v},\div\boldsymbol{\tau})_T$ requires the interior moments of the full space $\mathbb P_{k-1}(T; \mathbb R^d)$. This property holds for RT$_k$, but not for BDM$_k$.

\subsection{Weak div stability}
We first establish the following inf-sup condition.
\begin{theorem}\label{thm:infsupcurldiv}
There exists $\alpha>0$, independent of $h$, such that
\begin{equation}\label{eq:infsupcurldiv}
\inf_{(\boldsymbol{v}_h,\boldsymbol{\mu}_h)\in \mathring{\mathbb V}_{k,k-1}^{\div}\times \mathring{\Lambda}_k} 
\sup_{\boldsymbol{\tau}_h\in \Sigma_{k}^{-1}(\mathbb T)} 
\frac{(\div_w \boldsymbol{\tau}_h, (\boldsymbol{v}_h, \boldsymbol{\mu}_h))_{0,h}}
{\|\boldsymbol{\tau}_h\|_{\div_w}\,\|(\boldsymbol{v}_h, \boldsymbol{\mu}_h)\|_{0,h}} 
= \alpha.
\end{equation} 
\end{theorem}

\begin{proof}
Given $\boldsymbol{\mu}_h\in \mathring{\Lambda}_k$, construct $\boldsymbol{\tau}_1\in \Sigma_k^{-1}(\mathbb T)$ elementwise by
\[
\Pi_F\boldsymbol{\tau}_1\boldsymbol{n} = -\tfrac{1}{2}h_F\,\boldsymbol{\mu}_h 
\ \ \text{on } F\in\mathcal{F}(T), 
\qquad
(\boldsymbol{\tau}_1, \boldsymbol{q})_T = 0 
\ \ \forall\,\boldsymbol{q}\in \mathbb P_{k-1}(T;\mathbb T),
\]
for each $T\in\mathcal T_h$, where $\boldsymbol{n}$ is the unit outward normal on $\partial T$.  
Then $-[\![\Pi_F\boldsymbol{\tau}_1\boldsymbol{n}]\!] = h_F\boldsymbol{\mu}_h$. 
By inverse and scaling estimates,
\[
\|\boldsymbol{\tau}_1\|_{\div_w}
\lesssim \|h_F^{1/2}\boldsymbol{\mu}_h\|_{\mathcal F_h}
= \|(0, \boldsymbol{\mu}_h)\|_{0,h}.
\]
By the definition of $\div_w$,
\[
(\div_w \boldsymbol{\tau}_1, (0, \boldsymbol{\mu}_h))_{0,h} 
= - \sum_{F\in\mathring{\mathcal{F}}_h}([\![\Pi_F\boldsymbol{\tau}_1\boldsymbol{n}]\!], \boldsymbol{\mu}_h)_F
= \|(0, \boldsymbol{\mu}_h)\|_{0,h}^2,
\]
and hence
\[
\|(0, \boldsymbol{\mu}_h)\|_{0,h}
\lesssim \sup_{\boldsymbol{\tau}_h\in \Sigma_{k}^{-1}(\mathbb T)} 
\frac{(\div_w \boldsymbol{\tau}_h, (0, \boldsymbol{\mu}_h))_{0,h}}
{\|\boldsymbol{\tau}_h\|_{\div_w}} 
\qquad \forall\,\boldsymbol{\mu}_h\in \mathring{\Lambda}_k.
\]

On the other hand, by the weak divergence inf-sup condition (cf.~\eqref{eq:infsupcurldivw}),
\[
\|(\boldsymbol{v}_h, 0)\|_{0,h}=\|\boldsymbol{v}_h\|
\lesssim \sup_{\boldsymbol{\tau}_h\in \Sigma_{k}^{-1}(\mathbb T)} 
\frac{(\div_w \boldsymbol{\tau}_h, (\boldsymbol{v}_h, \boldsymbol{\mu}_h))_{0,h}}
{\|\boldsymbol{\tau}_h\|_{\div_w}} 
\qquad \forall\,\boldsymbol{v}_h\in \mathring{\mathbb V}_{k,k-1}^{\div},\ \boldsymbol{\mu}_h\in\mathring{\Lambda}_k.
\]
Combining the two bounds yields \eqref{eq:infsupcurldiv}.
\end{proof}

We use $(\cdot,\cdot)_{0,h}$ as a Riesz map to identify 
$E(\mathring{\mathbb V}_{k,\ell}^{\div} \times \mathring{\Lambda}_k)$ with its dual. 
The bilinear form \eqref{eq:divwT} thus induces a mapping 
\[
\div_w: \Sigma_k^{-1}(\mathbb T)\to \mathring{\mathbb V}_{k,\ell}^{\div} \times \mathring{\Lambda}_k.
\] 
The inf--sup condition \eqref{eq:infsupcurldiv} ensures that the weak divergence operator $\div_w$ is surjective when $\ell = k-1$. Consequently, its adjoint, the weak deviatoric gradient $\dev\grad_w$, is injective, and the mapping
$$ \| \dev\grad_w (\cdot, \cdot) \| $$
defines a norm on the product space $\mathring{\mathbb{V}}_{k,k-1}^{\div} \times \mathring{\Lambda}_k$. 

Again, it is important to note that the inf-sup condition \eqref{eq:infsupcurldiv} may not hold for the larger space $\mathring{\mathbb{V}}_{k,k}^{\div}$, even though this space is important for its approximation properties. In the subsequent analysis, the coercivity of the weak operators on the divergence-free subspaces is used to resolve this conflict.

\subsection{Mixed finite element methods}
With the operator $\dev \grad_w$, 
we formulate the following mixed finite element method for the Stokes equation~\eqref{eq:stokes}:  
Find $\boldsymbol{u}_h\in \mathring{\mathbb{V}}_{k,\ell}^{\div}$, 
$\boldsymbol{\lambda}_h\in\mathring{\Lambda}_k$, and 
$p_h\in\mathbb{P}_{\ell}(\mathcal{T}_h)/\mathbb R$, 
for integers $k \ge 0$ and $\ell\in\{k,\,k-1\}$, such that
\begin{subequations}\label{distribustokesfemWG}
\begin{align}
\label{distribustokesfemWG1}
(\dev \grad_w(\boldsymbol{u}_h,\boldsymbol{\lambda}_h),\, 
\dev \grad_w(\boldsymbol{v}_h,\boldsymbol{\mu}_h))
+ (\div\boldsymbol{v}_h,\, p_h)
&= (\boldsymbol{f},\, \boldsymbol{v}_h), \\
\label{distribustokesfemWG2}
(\div\boldsymbol{u}_h,\, q_h) &= 0,
\end{align}
\end{subequations}
for all 
$\boldsymbol{v}_h\in \mathring{\mathbb{V}}_{k,\ell}^{\div}$, 
$\boldsymbol{\mu}_h\in\mathring{\Lambda}_k$, and 
$q_h\in\mathbb{P}_{\ell}(\mathcal{T}_h)/\mathbb R$.

\begin{theorem}
The mixed finite element method~\eqref{distribustokesfemWG} is well-posed, 
with a unique solution 
$(\boldsymbol{u}_h,\, \boldsymbol{\lambda}_h,\, p_h)
\in \mathring{\mathbb{V}}_{k,\ell}^{\div}\times
\mathring{\Lambda}_k\times
\mathbb{P}_{\ell}(\mathcal{T}_h)/\mathbb R$.  
Let $\boldsymbol{\sigma}_h := \dev \grad_w(\boldsymbol{u}_h,\boldsymbol{\lambda}_h)$; 
then $\boldsymbol{\sigma}_h\in \Sigma_k^{\rm tn}$ and 
$\boldsymbol{u}_h\in \mathring{\mathbb V}_{k,\ell}^{\div}\cap \ker(\div)$ 
satisfy the mixed formulation~\eqref{eq:vecLapdivfree}.
\end{theorem}
%
\begin{proof}
We first show that the mixed method \eqref{distribustokesfemWG} admits only the zero solution when $\boldsymbol{f}=0$. From \eqref{distribustokesfemWG2}, it follows that $\boldsymbol{u}_h\in \mathring{\mathbb V}_{k,\ell}^{\div}\cap \ker(\div)
= \mathring{\mathbb V}_{k,k-1}^{\div}\cap \ker(\div)$.
Taking $\boldsymbol{v}_h=\boldsymbol{u}_h$ and $\boldsymbol{\mu}_h=\boldsymbol{\lambda}_h$ in \eqref{distribustokesfemWG1}, we obtain $\dev \grad_w(\boldsymbol{u}_h,\boldsymbol{\lambda}_h)=0$.
Combining this with \eqref{eq:infsupcurldiv} gives $\boldsymbol{u}_h=0$ and $\boldsymbol{\lambda}_h=0$. 
Furthermore, by $\div\mathring{\mathbb{V}}_{k,\ell}^{\div}=\mathbb{P}_{\ell}(\mathcal{T}_h)/\mathbb R$, equation \eqref{distribustokesfemWG1} implies $p_h=0$. Therefore, the mixed finite element method \eqref{distribustokesfemWG} is well-posed.

Since $\div\mathring{\mathbb{V}}_{k,\ell}^{\div}=\mathbb{P}_{\ell}(\mathcal{T}_h)/\mathbb R$, 
equation~\eqref{distribustokesfemWG2} implies that 
$\boldsymbol{u}_h\in \mathring{\mathbb{V}}_{k,\ell}^{\div}\cap \ker(\div)$.  
Restricting $\boldsymbol{v}_h$ to this subspace, 
equation~\eqref{distribustokesfemWG1} becomes
\begin{equation}\label{eq:20251009}
(\boldsymbol{\sigma}_h,\, \dev \grad_w(\boldsymbol{v}_h,\boldsymbol{\mu}_h))
= (\boldsymbol{f},\, \boldsymbol{v}_h),
\quad\forall\,
\boldsymbol{v}_h\in \mathring{\mathbb{V}}_{k,\ell}^{\div}\cap \ker(\div),
\ \boldsymbol{\mu}_h\in\mathring{\Lambda}_k.
\end{equation}
Taking $\boldsymbol{v}_h=0$ in~\eqref{eq:20251009} yields 
$\boldsymbol{\sigma}_h\in \Sigma_k^{\rm tn}$.  
Then~\eqref{eq:20251009} reduces to equation \eqref{eq:vecLapdivfree2}, 
while equation \eqref{eq:vecLapdivfree1} follows directly from the definition of $\boldsymbol{\sigma}_h$.
\end{proof}

We circumvent the inf-sup condition, as the error analysis will be derived through the equivalence to~\eqref{eq:vecLapdivfree} without relying on the inf-sup constant.
\subsection{Error analysis}
Hereafter, let $\boldsymbol{u} \in H_0^1(\Omega;\mathbb{R}^d)$ denote the solution of the Stokes equation \eqref{eq:stokes}, and set $\boldsymbol{\sigma} = \grad\boldsymbol{u} = \dev\grad\boldsymbol{u}$. 
Then the first equation in \eqref{eq:stokes} becomes
\begin{equation}\label{stokes1}
-\div\boldsymbol{\sigma}- \nabla p=\boldsymbol{f}  \quad \text { in } \Omega.
\end{equation}
Moreover, let $\boldsymbol{\lambda} \in L^2(\mathcal{F}_h; \mathbb{R}^{d-1})$ be defined facewise by
\[
\boldsymbol{\lambda}|_F = \Pi_F \boldsymbol{u}, \quad \forall\,F \in \mathcal{F}_h.
\]
This notation will be used throughout the subsequent analysis.

\begin{theorem}
Let $(\boldsymbol{u}, p)$ denote the solution of the Stokes equation \eqref{eq:stokes}, and let $(\boldsymbol{u}_h$, $\boldsymbol{\lambda}_h$, $p_h)$ be the solution of the mixed finite element method~\eqref{distribustokesfemWG}.
Assume $\boldsymbol{u}\in H^{k+2}(\Omega;\mathbb R^d)$. Then
\begin{align}
\label{eq:errorestimate1}
\|\boldsymbol{\sigma} - \boldsymbol{\sigma}_h\|_{0,h} + \| \dev\grad_w(I_{k,k}^{\div}\boldsymbol{u} - \boldsymbol{u}_h,Q_{k,\mathcal{F}_h}\boldsymbol{\lambda}-\boldsymbol{\lambda}_h)\| & \\
\notag
 + \|Q_{\ell} p - p_h\| &\lesssim h^{k+1} |\boldsymbol{u}|_{k+2},
 \end{align}
where $\|\boldsymbol{\tau}\|_{0,h}^2:=\|\boldsymbol{\tau}\|^2+\sum_{F\in\mathcal{F}_h}h_F\|\Pi_F\boldsymbol{\tau}\boldsymbol{n}\|_F^2$.
\end{theorem}
\begin{proof}
By the equivalence of \eqref{distribustokesfemWG} and \eqref{eq:vecLapdivfree}, the estimate of $\|\boldsymbol{\sigma} - \boldsymbol{\sigma}_h\|_{0,h}$ follows from \eqref{eq:uestimate} in Theorem \ref{thm:divfreemixfem}. 

Then by Lemma \ref{lm:weakdevgrad},
\[
\begin{aligned}
\dev\grad_w(I_{k,k}^{\div}\boldsymbol{u},Q_{k,\mathcal{F}_h}\boldsymbol{\lambda}) &= Q_k\dev\grad \bs u = Q_k \bs \sigma,\\
\dev\grad_w(\bs u_h, \bs \lambda_h) &= \bs \sigma_h.
\end{aligned}
\]
So 
\[
\| \dev\grad_w(I_{k,k}^{\div}\boldsymbol{u} - \boldsymbol{u}_h,Q_{k,\mathcal{F}_h}\boldsymbol{\lambda}-\boldsymbol{\lambda}_h)\| = \|Q_k\bs \sigma - \bs \sigma_h\|\leq \|\bs \sigma - \bs \sigma_h\|. 
\]

We now estimate the error of the pressure approximation. For $\boldsymbol{v}\in H_0^1(\Omega;\mathbb R^d)$, by \eqref{stokes1},
\begin{align*}
(\div(I_{k,\ell}^{\div}\boldsymbol{v}),p-p_h) &= (\div(I_{k,\ell}^{\div}\boldsymbol{v}),p)-(\boldsymbol{f}, I_{k,\ell}^{\div}\boldsymbol{v}) + (\boldsymbol{\sigma}_h, \dev\grad_w(I_{k,\ell}^{\div}\boldsymbol{v},0)) \\
&= (\div\boldsymbol{\sigma}, I_{k,\ell}^{\div}\boldsymbol{v}) + (\boldsymbol{\sigma}_h, \dev\grad_w(I_{k,\ell}^{\div}\boldsymbol{v},0)) \\
&= (I_k^{\rm tn}\bs \sigma-\boldsymbol{\sigma}, \grad_h(I_{k,\ell}^{\div}\boldsymbol{v})) \!- \!\!\!\sum_{T\in\mathcal{T}_h}(\Pi_F(I_k^{\rm tn}\bs \sigma-\boldsymbol{\sigma})\boldsymbol{n},\Pi_F(I_{k,\ell}^{\div}\boldsymbol{v}))_{\partial T}\\
&\quad +(\boldsymbol{\sigma}_h-I_k^{\rm tn}\bs \sigma, \dev\grad_w(I_{k,\ell}^{\div}\boldsymbol{v},0)).
\end{align*}
Then we get from \eqref{eq:Ihtnprop2} and \eqref{eq:Ihdivprop2} that
\[
(\div(I_{k,\ell}^{\div}\boldsymbol{v}),p-p_h)\lesssim h^{k+1} |\boldsymbol{u}|_{k+2}|\boldsymbol{v}|_1.
\]
Finally, by the inf-sup condition for $\div$ operator, we have
\[
\|Q_{\ell} p - p_h\|\lesssim \sup_{\boldsymbol{v}\in H_0^1(\Omega;\mathbb R^d)}\frac{(\div(I_{k,\ell}^{\div}\boldsymbol{v}),p-p_h)}{|\boldsymbol{v}|_1} \lesssim h^{k+1} |\boldsymbol{u}|_{k+2}.
\]
\end{proof}

The estimate \eqref{eq:errorestimate1} demonstrates that the mixed finite element method \eqref{distribustokesfemWG} is pressure-robust as it depends only on the regularity of $\bs u$ not $p$.


\subsection{$L^2$ error estimate}
The superconvergence of $\|I_{k,k}^{\div}\boldsymbol{u}-\boldsymbol{u}_h\|$ can be obtained through the duality argument.
Introduce the dual problem: Find $\tilde{\bs u}\in H_0^1(\Omega;\mathbb R^d)$ and $\tilde p\in L^2(\Omega)/\mathbb R$ that satisfy 
\begin{equation}\label{dualstokes}
\begin{cases}
- \Delta \tilde{\bs u} - \nabla \tilde{p}= I_{k,k}^{\div}\boldsymbol{u}-\boldsymbol{u}_h& \text { in } \Omega, \\
\qquad\;\;\div\tilde{\boldsymbol{u}} =0 & \text { in } \Omega.
\end{cases}
\end{equation}
We assume that the dual problem \eqref{dualstokes} has
the $H^2$-regularity 
\begin{equation}\label{eq:regularity}
\|\tilde{\boldsymbol{u}}\|_2 \lesssim \|I_{k,k}^{\div}\boldsymbol{u}-\boldsymbol{u}_h\|.
\end{equation}
We refer to \cite[Remark I.5.6]{GiraultRaviart1986} and \cite[Section 11.5]{MazyaRossmann2010} for the $H^2$-regularity \eqref{eq:regularity} in convex domains.

\begin{lemma}
Let $(\boldsymbol{u}, p)$ denote the solution of the Stokes equation \eqref{eq:stokes}, and let $(\boldsymbol{u}_h$, $\boldsymbol{\lambda}_h$, $p_h)$ be the solution of the mixed finite element method~\eqref{distribustokesfemWG}. Let $\tilde{\bs u}$ be the solution of the dual problem \eqref{dualstokes}. Assume $\skw\grad\boldsymbol{f}\in L^2(\Omega;\mathbb K)$. We have
\begin{equation}
\label{eq:errorestimate40}
(\boldsymbol{\sigma}-\boldsymbol{\sigma}_h,\grad\tilde{\boldsymbol{u}}) \lesssim \begin{cases}
h\|\boldsymbol{\sigma}-\boldsymbol{\sigma}_h\|_{0,h} |\tilde{\boldsymbol{u}}|_{2}, & k\geq1,   \\
h^2\|\skw\grad\boldsymbol{f}\|\|\tilde{\boldsymbol{u}}\|_1, & k=0.
\end{cases}
\end{equation}
\end{lemma}
\begin{proof}
It follows from \eqref{distribustokesfemWG1} with $\boldsymbol{v}_h = I_{k,k}^{\div}\tilde{\boldsymbol{u}}$ and $\boldsymbol{\mu}_h = 0$, $\div(I_{k,k}^{\div}\tilde{\boldsymbol{u}})=0$, and \eqref{stokes1} that
\begin{equation*}
(\boldsymbol{\sigma}_h, \dev \grad_w(I_{k,k}^{\div}\tilde{\boldsymbol{u}},0)) = (\boldsymbol{f}, I_{k,k}^{\div}\tilde{\boldsymbol{u}}) = -(\div\boldsymbol{\sigma}, I_{k,k}^{\div}\tilde{\boldsymbol{u}}).
\end{equation*}
That is, 
\[
\sum_{T\in\mathcal{T}_h}(\boldsymbol{\sigma}-\boldsymbol{\sigma}_h,\dev\grad(I_{k,k}^{\div}\tilde{\boldsymbol{u}}))_T - \sum_{F\in\mathcal{F}_h}(\Pi_F(\boldsymbol{\sigma}-\boldsymbol{\sigma}_h)\boldsymbol{n}, [\![\Pi_F(I_{k,k}^{\div}\tilde{\boldsymbol{u}})]\!])_{F}=0.
\]
Then
\begin{equation}\label{eq:20250923}
\begin{aligned}
(\boldsymbol{\sigma}-\boldsymbol{\sigma}_h,\grad\tilde{\boldsymbol{u}}) &= \sum_{T\in\mathcal{T}_h}(\boldsymbol{\sigma}-\boldsymbol{\sigma}_h,\dev\grad(\tilde{\boldsymbol{u}}-I_{k,k}^{\div}\tilde{\boldsymbol{u}}))_T \\
&\quad - \sum_{F\in\mathcal{F}_h}(\Pi_F(\boldsymbol{\sigma}-\boldsymbol{\sigma}_h)\boldsymbol{n}, [\![\Pi_F(\tilde{\boldsymbol{u}}-I_{k,k}^{\div}\tilde{\boldsymbol{u}})]\!])_{F}.
\end{aligned}
\end{equation}
When $k\geq1$, we get from the Cauchy-Schwarz inequality and the estimate \eqref{eq:Ihdivprop2} of $I_{k,k}^{\div}$ that
\begin{equation*}
(\boldsymbol{\sigma}-\boldsymbol{\sigma}_h,\grad\tilde{\boldsymbol{u}}) \lesssim h\|\boldsymbol{\sigma}-\boldsymbol{\sigma}_h\|_{0,h} |\tilde{\boldsymbol{u}}|_{2}.
\end{equation*}

Next consider case $k=0$. Applying integration by parts to the right hand side of \eqref{eq:20250923} to have
\begin{equation*}
(\boldsymbol{\sigma}-\boldsymbol{\sigma}_h,\grad\tilde{\boldsymbol{u}}) =-(\div\boldsymbol{\sigma},\tilde{\boldsymbol{u}}-I_{0,0}^{\div}\tilde{\boldsymbol{u}}).
\end{equation*}
Since $\div\tilde{\boldsymbol{u}}=0$ and $\div H_0^2(\Omega;\mathbb K)=H_0^1(\Omega;\mathbb R^d)\cap \ker(\div)$ (cf. \cite[Theorem 1.1]{CostabelMcIntosh2010}), 
there exists a $\tilde{\boldsymbol{\tau}}\in H_0^2(\Omega;\mathbb K)$ such that
\begin{equation*}
\div\tilde{\boldsymbol{\tau}}=\tilde{\boldsymbol{u}},\quad\textrm{and}\quad \|\tilde{\boldsymbol{\tau}}\|_2\lesssim \|\tilde{\boldsymbol{u}}\|_1.
\end{equation*}
Recall the linear finite element space of the $(d-2)$-form in skew-symmetric form \cite{Arnold2018,ArnoldFalkWinther2006}
\begin{equation*}
\mathbb V_{h}^{d-2}:=\{\boldsymbol{\tau}_h\in H(\div,\Omega;\mathbb K): \boldsymbol{\tau}_h|_T \in \mathbb P_{1}(T;\mathbb K) \;\;\textrm{ for } T \in \mathcal{T}_{h}\}.
\end{equation*}
The local DoFs for space $\mathbb V_{h}^{d-2}$ are given in \cite[(3.10)]{ChenHuangWei2024}. Let $I_h^{d-2}: H^2(\Omega;\mathbb K)\to \mathbb V_{h}^{d-2}$ be the nodal interpolation operator.
It holds the following commuting property:
\begin{equation*}
\div(I_h^{d-2}\boldsymbol{\tau})=I_{0,0}^{\div}(\div\boldsymbol{\tau})\quad\forall~\boldsymbol{\tau}\in H^2(\Omega;\mathbb K).
\end{equation*}
Then, employing integration by parts and \eqref{stokes1},
\begin{align*}
(\boldsymbol{\sigma}-\boldsymbol{\sigma}_h,\grad\tilde{\boldsymbol{u}}) &= -(\div\boldsymbol{\sigma},\div(\tilde{\boldsymbol{\tau}}-I_{h}^{d-2}\tilde{\boldsymbol{\tau}})) = (\skw\grad(\div\boldsymbol{\sigma}), \tilde{\boldsymbol{\tau}}-I_{h}^{d-2}\tilde{\boldsymbol{\tau}}) \\
& = -(\skw\grad\boldsymbol{f}, \tilde{\boldsymbol{\tau}}-I_{h}^{d-2}\tilde{\boldsymbol{\tau}}) .
\end{align*}
This together with the interpolation error estimate of $I_{h}^{d-2}$ gives
\begin{equation*}
(\boldsymbol{\sigma}_h-\boldsymbol{\sigma},\grad\tilde{\boldsymbol{u}})\lesssim h^2\|\skw\grad\boldsymbol{f}\||\tilde{\boldsymbol{\tau}}|_2\lesssim h^2\|\skw\grad\boldsymbol{f}\|\|\tilde{\boldsymbol{u}}\|_1.
\end{equation*}
This ends the proof.
\end{proof}

\begin{theorem}
Let $(\boldsymbol{u}, p)$ denote the solution of the Stokes equation \eqref{eq:stokes}, and let $(\boldsymbol{u}_h$, $\boldsymbol{\lambda}_h$, $p_h)$ be the solution of the mixed finite element method~\eqref{distribustokesfemWG}. Assume $\boldsymbol{u}\in H^{k+2}(\Omega;\mathbb R^d)$, $\skw\grad\boldsymbol{f}\in L^2(\Omega;\mathbb K)$, and the $H^2$-regularity \eqref{eq:regularity} holds. Then
\begin{equation}
\label{eq:errorestimate4}
\|I_{k,k}^{\div}\boldsymbol{u} - \boldsymbol{u}_h\| \lesssim h^{k+2}(|\boldsymbol{u}|_{k+2}  + \delta_{k0}\|\skw\grad\boldsymbol{f}\|),
\end{equation}
where $\delta_{00}=1$ and $\delta_{k0}=0$ for $k\geq1$.
\end{theorem}
\begin{proof}
Set $\boldsymbol{v}_h=I_{k,k}^{\div}\boldsymbol{u}-\boldsymbol{u}_h\in\mathring{\mathbb V}_{k,k-1}^{\div}\cap\ker(\div)$ for ease of presentation.
Let $\tilde{\bs \sigma} = \grad \tilde{\bs u}$. By the fact $\div\boldsymbol{v}_h=0$, and the definition of the operator $I_k^{\rm tn}$, we have
\begin{align*}
\|\boldsymbol{v}_h\|^2 &= -(\div \tilde{\boldsymbol{\sigma}}, \boldsymbol{v}_h) = \sum_{T\in\mathcal{T}_h}(\tilde{\boldsymbol{\sigma}},\dev \grad\boldsymbol{v}_h)_T - \sum_{T\in\mathcal{T}_h}(\Pi_F\boldsymbol{v}_h, \Pi_F\tilde{\boldsymbol{\sigma}}\boldsymbol{n})_{\partial T} \\
& = \sum_{T\in\mathcal{T}_h}(I_k^{\rm tn}\tilde{\boldsymbol{\sigma}},\dev \grad\boldsymbol{v}_h)_T -\sum_{T\in\mathcal{T}_h}(\Pi_F\boldsymbol{v}_h, \Pi_F(I_k^{\rm tn}\tilde{\boldsymbol{\sigma}})\boldsymbol{n})_{\partial T} \\
&=\big(I_k^{\rm tn}\tilde{\boldsymbol{\sigma}},\dev \grad_w(I_{k,k}^{\div}\boldsymbol{u}-\boldsymbol{u}_h,Q_{k,\mathcal{F}_h}\boldsymbol{\lambda}-\boldsymbol{\lambda}_h)\big).
\end{align*}
This combined with \eqref{eq:weakdevgradcommu} implies
\[
\begin{aligned}
\|I_{k,k}^{\div}\boldsymbol{u}-\boldsymbol{u}_h\|^2& = (\boldsymbol{\sigma}-\boldsymbol{\sigma}_h,I_k^{\rm tn}\tilde{\boldsymbol{\sigma}})\\
&=(\boldsymbol{\sigma}-\boldsymbol{\sigma}_h,I_k^{\rm tn}\tilde{\boldsymbol{\sigma}}-\tilde{\boldsymbol{\sigma}}) + (\boldsymbol{\sigma}-\boldsymbol{\sigma}_h,\grad\tilde{\boldsymbol{u}}).
\end{aligned}
\]
Then we get from the Cauchy-Schwarz inequality, the estimate \eqref{eq:Ihtnprop2} of $I_k^{\rm tn}$, \eqref{eq:errorestimate40} and the $H^2$-regularity \eqref{eq:regularity} that 
\begin{equation*}
\|I_{k,k}^{\div}\boldsymbol{u}-\boldsymbol{u}_h\|\lesssim h\|\boldsymbol{\sigma}-\boldsymbol{\sigma}_h\|_{0,h} + \delta_{k0}h^2\|\skw\grad\boldsymbol{f}\|.
\end{equation*}
Thus, the estimate \eqref{eq:errorestimate4} follows from \eqref{eq:errorestimate1} and the last inequality.
\end{proof}

\subsection{Postprocessing}
We can construct a superconvergent approximation of $\boldsymbol{u}$ using the estimates \eqref{eq:errorestimate1} and \eqref{eq:errorestimate4}.  
For each $T \in \mathcal{T}_h$, find $\boldsymbol{u}_h^{*}|_T\in \mathbb{P}_{k+1}(T;\mathbb{R}^d)$ and $p_h^{*}\in \mathbb P_k(T)/\mathbb{R}$ such that
\begin{subequations}\label{distribustokespostfem}
\begin{align}
(\boldsymbol{u}_h^{*}\cdot\boldsymbol{n},q)_F &=
(\boldsymbol{u}_h\cdot \boldsymbol{n},q)_F 
\quad\;\;\; \forall\,q\in\mathbb{P}_0(F),\ F\in\mathcal{F}(T), \label{eq:postprocessingu1}\\
(\grad\boldsymbol{u}_h^{*},\grad\boldsymbol{v})_T
+(\div\boldsymbol{v},p_h^{*})_T &=
(\boldsymbol{\sigma}_h,\grad\boldsymbol{v})_T 
\quad \forall\,\boldsymbol{v}\in\widetilde{\mathbb{P}}_{k+1}(T;\mathbb{R}^d), \label{eq:postprocessingu2}\\
(\div \boldsymbol{u}_h^{*},q)_T &= 0 
\qquad\qquad\qquad\; \forall\,q\in \mathbb P_k(T)/\mathbb{R}, \label{eq:postprocessingu3}
\end{align}
\end{subequations}
where 
$\widetilde{\mathbb{P}}_{k+1}(T;\mathbb{R}^d)
:= \{\boldsymbol{v}_h\in\mathbb{P}_{k+1}(T;\mathbb{R}^d): 
Q_{0,F}(\boldsymbol{v}_h\cdot\boldsymbol{n})=0\ \ \forall\,F\subset\partial T\}$.

The local problems \eqref{distribustokespostfem} are well-posed. From \eqref{eq:postprocessingu1}, it follows that
$$ Q_{0,T}(\div\boldsymbol{u}_h^{*}) = Q_{0,T}(\div\boldsymbol{u}_h) = 0. $$
Combined with \eqref{eq:postprocessingu3}, this identity implies that $\div\boldsymbol{u}_h^{*} = 0$ for each $T \in \mathcal{T}_h$. Consequently, the postprocessed velocity $\boldsymbol{u}_h^{*} \in \mathbb{P}_{k+1}(\mathcal{T}_h; \mathbb{R}^d)$ is pointwise divergence-free on every element. However, we note that $\boldsymbol{u}_h^{*}$ is not globally $H(\div)$-conforming, as the normal trace continuity \eqref{eq:postprocessingu1} is  weakly enforced for $q \in \mathbb{P}_0(F)$ only.

\begin{theorem}
Let $(\boldsymbol{u}, p)$ be the solution of the Stokes equation~\eqref{eq:stokes}, and let $(\boldsymbol{u}_h$, $\boldsymbol{\lambda}_h$, $p_h)$ be the solution of the mixed finite element method~\eqref{distribustokesfemWG}. 
Assume $\boldsymbol{u}\in H^{k+2}(\Omega;\mathbb R^d)$. Then
\begin{equation}\label{eq:postH1error}
\|\grad_h(\boldsymbol{u}-\boldsymbol{u}_h^{*})\|\lesssim h^{k+1}|\boldsymbol{u}|_{k+2}.
\end{equation}
If, in addition, $\skw\grad\boldsymbol{f} \in L^2(\Omega; \mathbb{K})$ and the $H^2$-regularity estimate~\eqref{eq:regularity} holds, we have
\begin{equation}\label{eq:postL2error}
\|\boldsymbol{u}-\boldsymbol{u}_h^{*}\|\lesssim h^{k+2}\big(|\boldsymbol{u}|_{k+2}  + \delta_{k0}\|\skw\grad\boldsymbol{f}\|\big).
\end{equation}
\end{theorem}

\begin{proof}
Let 
\[
\boldsymbol{w}=(I-I_{(0,0),T}^{\div})(I_{(k+1,k),T}^{\div}\boldsymbol{u}-\boldsymbol{u}_h^{*})\in\widetilde{\mathbb{P}}_{k+1}(T;\mathbb{R}^d).
\]
By~\eqref{eq:Ihdivprop1},
\[
\div\big(I_{(0,0),T}^{\div}(I_{(k+1,k),T}^{\div}\boldsymbol{u}-\boldsymbol{u}_h^{*})\big)
=Q_{0,T}\div\boldsymbol{u}-Q_{0,T}\div\boldsymbol{u}_h^{*}=0.
\]
Hence $I_{(0,0),T}^{\div}(I_{(k+1,k),T}^{\div}\boldsymbol{u}-\boldsymbol{u}_h^{*})\in\mathbb{P}_{0}(T;\mathbb{R}^d)$, and applying~\eqref{eq:Ihdivprop1} again gives
\[
\div\boldsymbol{w}
=\div(I_{(k+1,k),T}^{\div}\boldsymbol{u}-\boldsymbol{u}_h^{*})
=Q_{k,T}\div\boldsymbol{u}-\div\boldsymbol{u}_h^{*}=0.
\]
Using~\eqref{eq:postprocessingu2} with $\boldsymbol{v}=\boldsymbol{w}$ yields
\begin{align*}
|\boldsymbol{w}|_{1,T}^2
&=(\grad(I_{(k+1,k),T}^{\div}\boldsymbol{u}-\boldsymbol{u}_h^{*}),\grad\boldsymbol{w})_T \\
&=(\grad(I_{(k+1,k),T}^{\div}\boldsymbol{u}-\boldsymbol{u}),\grad\boldsymbol{w})_T
+(\boldsymbol{\sigma}-\boldsymbol{\sigma}_h,\grad\boldsymbol{w})_T.
\end{align*}
Applying the interpolation estimate~\eqref{eq:Ihdivprop2} and the inverse inequality gives
\begin{equation}\label{eq:20250703}
\|\boldsymbol{w}\|_T\eqsim h_T|I_{(k+1,k),T}^{\div}\boldsymbol{u}-\boldsymbol{u}_h^{*}|_{1,T}
\lesssim h_T\big(|\boldsymbol{u}-I_{(k+1,k),T}^{\div}\boldsymbol{u}|_{1,T}+\|\boldsymbol{\sigma}-\boldsymbol{\sigma}_h\|_T\big).
\end{equation}
Thus, the estimate~\eqref{eq:postH1error} follows from the triangle inequality, 
the interpolation bound~\eqref{eq:Ihdivprop2}, 
and the stress estimate~\eqref{eq:errorestimate1}.

From~\eqref{eq:postprocessingu1},
\[
I_{(0,0),T}^{\div}(I_{(k+1,k),T}^{\div}\boldsymbol{u}-\boldsymbol{u}_h^{*})
=I_{(0,0),T}^{\div}(I_{(k,k),T}^{\div}\boldsymbol{u}-\boldsymbol{u}_h).
\]
Using the triangle inequality,~\eqref{eq:Ihdivprop2}, the inverse inequality, and~\eqref{eq:20250703}, we obtain
\begin{align*}
\|\boldsymbol{u}-\boldsymbol{u}_h^{*}\|_T
&\le \|\boldsymbol{u}-I_{(k+1,k),T}^{\div}\boldsymbol{u}\|_T
+\|I_{(0,0),T}^{\div}(I_{(k,k),T}^{\div}\boldsymbol{u}-\boldsymbol{u}_h)\|_T
+\|\boldsymbol{w}\|_T\\
&\lesssim \|\boldsymbol{u}-I_{(k+1,k),T}^{\div}\boldsymbol{u}\|_T
+\|I_{(k,k),T}^{\div}\boldsymbol{u}-\boldsymbol{u}_h\|_T\\
&\quad +h_T\big(|\boldsymbol{u}-I_{(k+1,k),T}^{\div}\boldsymbol{u}|_{1,T}
+\|\boldsymbol{\sigma}-\boldsymbol{\sigma}_h\|_T\big).
\end{align*}
Combining~\eqref{eq:Ihdivprop2},~\eqref{eq:errorestimate4}, and~\eqref{eq:errorestimate1} gives~\eqref{eq:postL2error}.
\end{proof}

\begin{remark}\rm
The estimate for $\|Q_{\ell}p-p_h\|$ in \eqref{eq:errorestimate1} is optimal when $\ell=k$ and superconvergent when $\ell=k-1$. For fixed $k\ge 0$, since
\[
\mathring{\mathbb{V}}_{k,k}^{\div}\cap \ker(\div)
=\mathring{\mathbb{V}}_{k,k-1}^{\div}\cap \ker(\div),
\]
the discrete solutions $\boldsymbol{u}_h$ and $\boldsymbol{\lambda}_h$ are independent of $\ell$. In fact, if $(\boldsymbol{u}_h,\boldsymbol{\lambda}_h,p_h)$ solves the mixed method \eqref{distribustokesfemWG} with $\ell=k$, then $(\boldsymbol{u}_h,\boldsymbol{\lambda}_h,Q_{k-1}p_h)$ is the solution of the method with $\ell=k-1$.
 When $k\ge 1$, we have $p_h-Q_{k-1}p_h\in \mathbb{P}_k(T)/\mathbb{R}$ and
\[
\div\big(\mathbb{V}_{k,k}^{\div}(T)\cap H_0(\div,T)\big)=\mathbb{P}_k(T)/\mathbb{R},
\]
so $p_h$ can be recovered locally from $Q_{k-1}p_h$. Specifically, for each $T\in\mathcal{T}_h$, $p_h-Q_{k-1}p_h$ is characterized by
\[
(\div\boldsymbol{v}_h, p_h-Q_{k-1}p_h)_T
= (\boldsymbol{f}, \boldsymbol{v}_h)_T-(\boldsymbol{\sigma}_h, \dev \grad_w(\boldsymbol{v}_h,0))_T-(\div\boldsymbol{v}_h, Q_{k-1}p_h)_T,
\]
for all $\boldsymbol{v}_h\in \mathbb{V}_{k,k}^{\div}(T)\cap H_0(\div,T)$.
We refer to \cite{BeiraodaVeigaLovadinaVacca2017,WeiHuangLi2021,HuangWang2023} for related pressure post-processing procedures in virtual element methods.
\end{remark}

\section{Equivalent Discrete Methods}\label{sec:hybrid}
We relate our mixed finite element method to several existing divergence-free schemes. First, hybridizing the velocity yields an equivalent $H^1$-nonconforming virtual element formulation, which simplifies implementation in low-order cases. Second, we show equivalence to stabilization-free virtual element methods, where the weak deviatoric gradient serves as a local projection. Finally, we compare our method with the Mass-Conserving-Stress (MCS) method \cite{GopalakrishnanLedererSchoeberl2020,GopalakrishnanLedererSchoeberl2020a} and show that our framework naturally covers low-order cases not included there.

\subsection{Hybridization of the velocity}

We recall the $H^1$-nonconforming virtual element in \cite{AyusodeDiosLipnikovManzini2016,ChenHuang2020}. 
For integers $k\geq0$ and $\ell=k,k-1$, the shape function space is defined as
\begin{equation}
\label{ncve}
V_{k+1,\ell+2}^{\rm VE}(T)=\{v\in H^1(T):\Delta v\in\mathbb{P}_{\ell}(T), \partial_n v|_F\in\mathbb{P}_{k}(F) \textrm{ for } F\in\mathcal{F}(T)\}.
\end{equation}
The DoFs for $V_{k+1,\ell+2}^{\rm VE}(T)$ are given by
\begin{subequations}\label{ve-dof}
	\begin{align}
		\label{ve-dof1}
		(v,q)_F, & \quad q\in \mathbb{P}_{k}(F), \, F\in\mathcal{F}(T),\\
		\label{ve-dof2}
		(v,q)_T, & \quad q\in \mathbb{P}_{\ell}(T).
	\end{align}
\end{subequations}
The global virtual element space $V_{k+1,\ell+2}^{\rm VE} = V_{k+1,\ell+2}^{\rm VE}(\mathcal T_h)$ by
\begin{align*}
V_{k+1,\ell+2}^{\rm VE}(\mathcal T_h) &= \{v\in L^2(\Omega):v|_T\in V_{k+1,\ell+2}^{\rm VE}(T) \textrm{ for each }T\in\mathcal{T}_h; \textrm{ DoF \eqref{ve-dof1} is } \\
&\qquad\qquad\qquad\qquad\qquad\quad\quad\;\;\textrm{ single-valued across each face in $\mathring{\mathcal{F}}_h$}\}.
\end{align*}
The virtual element space $V_{k+1,\ell+2}^{\rm VE}(\mathcal T_h)$ satisfies the weak continuity condition
\begin{equation*}
([\![v]\!], q)_F = 0 
\quad \forall\, v\in V_{k+1,\ell+2}^{\rm VE}(\mathcal T_h), \ 
q\in\mathbb{P}_{k}(F), \ 
F\in\mathring{\mathcal{F}}_h.
\end{equation*}
When $(k,\ell) = (0,-1)$, the space $V_{1,1}^{\rm VE}$ reduces to the Crouzeix--Raviart (CR) element~\cite{CrouzeixRaviart1973}.  
For $(k,\ell) = (0,0)$, the space $V_{1,2}^{\rm VE}$ is the enriched Crouzeix--Raviart element~\cite{HuMa2015,HuHuangLin2014}.

Find $\boldsymbol{u}_h\in \prod_{T\in \mathcal T_h}\mathbb V_{k,\ell}^{\div}(T)$, $\boldsymbol{\lambda}_h\in\mathbb P_{k}(\mathring{\mathcal{F}}_h;\mathbb R^{d-1})$ and $p_h\in V_{k+1,\ell+2}^{\rm VE}(\mathcal T_h)/\mathbb R$ such that
\begin{subequations}\label{distribustokesH1pfemWG}
\begin{align}
\label{distribustokesH1pfemWG1}
(\dev \grad_w(\boldsymbol{u}_h,\boldsymbol{\lambda}_h), \dev \grad_w(\boldsymbol{v}_h,\boldsymbol{\mu}_h)) - (\boldsymbol{v}_h, \nabla_hp_h) &= (\boldsymbol{f}, \boldsymbol{v}_h), \\
\label{distribustokesH1pfemWG2}
(\boldsymbol{u}_h, \nabla_hq_h)&=0
\end{align}
\end{subequations}
for all $\boldsymbol{v}_h\in \prod_{T\in \mathcal T_h}\mathbb V_{k,\ell}^{\div}(T)$, $\boldsymbol{\mu}_h\in\mathbb P_{k}(\mathring{\mathcal{F}}_h;\mathbb R^{d-1})$ and $q_h\in V_{k+1,\ell+2}^{\rm VE}(\mathcal T_h)/\mathbb R$. 

\begin{theorem}\label{thm:equivalence0}
The mixed method \eqref{distribustokesH1pfemWG} is well-posed.
Let $(\boldsymbol{u}_h, \boldsymbol{\lambda}_h, p_h)\in \prod_{T\in \mathcal T_h}$ $\mathbb V_{k,\ell}^{\div}(T)$ $\times \mathring{\Lambda}_k\times (V_{k+1,\ell+2}^{\rm VE}(\mathcal T_h)/\mathbb R)$ be its solution, then $(\boldsymbol{u}_h,\boldsymbol{\lambda}_h,Q_{\ell}p_h)\in \mathring{\mathbb{V}}_{k,\ell}^{\div}\times \mathring{\Lambda}_k\times(\mathbb{P}_{\ell}(\mathcal{T}_h)/\mathbb R)$ is the solution of the mixed finite element method \eqref{distribustokesfemWG}.
\end{theorem}
\begin{proof}
We first prove the discrete method \eqref{distribustokesH1pfemWG} has the zero solution when $\boldsymbol{f}=0$. 
By applying the integration by parts to \eqref{distribustokesH1pfemWG2},
\begin{equation*}
\sum_{T\in\mathcal{T}_h}(\div\boldsymbol{u}_h, Q_{\ell,T}q_h)_T - \sum_{F\in\mathcal{F}_h}([\![\boldsymbol{u}_h\cdot\boldsymbol{n}]\!], Q_{k,F}q_h)_F=0 \quad \forall~q_h\in V_{k+1,\ell+2}^{\rm VE}(\mathcal T_h).
\end{equation*}
Thanks to DoFs \eqref{ve-dof}, the above equation implies $\boldsymbol{u}_h\in\mathring{\mathbb{V}}_{k,\ell}^{\div}$ and $\div\boldsymbol{u}_h=0$.

Let $Q_T^{\rm div}: L^2(T; \mathbb{R}^d) \to \mathbb V_{k,\ell}^{\div}(T)$ be the $L^2$-orthogonal projection operator. 
Recall the following norm equivalences on each $T\in\mathcal{T}_h$ established in ~\cite{HuangTang2025,ChenHuang2020,ChenHuangWei2024}:
\begin{align}
\label{ve-infsup2}
\|Q_{T}^{\rm div}\nabla v\|_{0,T}&\eqsim\|\nabla v\|_{0,T} \quad \forall \ v\in V_{k+1,\ell+2}^{\rm VE}(T).
\end{align}

By adding equation \eqref{distribustokesH1pfemWG1} with $(\boldsymbol{v}_h,\boldsymbol{\mu}_h)=(\boldsymbol{u}_h,\boldsymbol{\lambda}_h)$ and equation \eqref{distribustokesH1pfemWG2} with $q_h=p_h$, we get $\dev \grad_w(\boldsymbol{u}_h,\boldsymbol{\lambda}_h)=0$, which means $\boldsymbol{u}_h=0$ and $\boldsymbol{\lambda}_h=0$. Then we get from equation \eqref{distribustokesH1pfemWG1} and the norm equivalence \eqref{ve-infsup2} that $p_h=0$.
Thus, the discrete method \eqref{distribustokesH1pfemWG} is well-posed.

For the second part, equation \eqref{distribustokesH1pfemWG2} indicates $\boldsymbol{u}_h\in\mathring{\mathbb{V}}_{k,\ell}^{\div}$ and $\div\boldsymbol{u}_h=0$.
By restricting $\boldsymbol{v}_h\in\mathring{\mathbb{V}}_{k,\ell}^{\div}$, it follows from the integration by parts that
\[
 - (\boldsymbol{v}_h, \nabla_hp_h) = (\div\boldsymbol{v}_h, Q_{\ell}p_h).
\]
Hence, $(\boldsymbol{u}_h,\boldsymbol{\lambda}_h,Q_{\ell}p_h)$ satisfies equation \eqref{distribustokesfemWG1}. This ends the proof.
\end{proof}


The normal continuity for the velocity space is relaxed in \eqref{distribustokesH1pfemWG} which is particularly useful when $\ell = -1$. As a local basis for the space $\mathring{\mathbb{V}}_{0,-1}^{\div}$ is not readily available, we can implement the equivalent mixed method \eqref{distribustokesH1pfemWG} instead of the mixed method \eqref{distribustokesfemWG}.

Even for $(k, \ell)=(0, -1)$, 
our method has the following estimates:
\begin{equation*}
\begin{aligned}
&\|\boldsymbol{\sigma} - \boldsymbol{\sigma}_h\|_{0,h} + \|\grad_h(\boldsymbol{u}-\boldsymbol{u}_h^{*})\| + \|\dev\grad_w(I_{0,0}^{\div}\boldsymbol{u} - \boldsymbol{u}_h,Q_{0,\mathcal{F}_h}\boldsymbol{\lambda}-\boldsymbol{\lambda}_h)\| \lesssim h |\boldsymbol{u}|_{2},\\
&\|\boldsymbol{u} - \boldsymbol{u}_h\|_0+h\|\grad_h(\boldsymbol{u} - \boldsymbol{u}_h)\| \lesssim h( |\boldsymbol{u}|_{2} + |\boldsymbol{u}|_{1}),\\	
&\|\boldsymbol{u}-\boldsymbol{u}_h^{*}\|+\|I_{0,0}^{\div}\boldsymbol{u} - \boldsymbol{u}_h\| \lesssim h^{2}(|\boldsymbol{u}|_{2}  + \|\skw\grad\boldsymbol{f}\|).
\end{aligned}
\end{equation*}

\subsection{Stabilization-free virtual element methods}
For $k\geq0$, define the space of vector-valued shape functions for an $H^1$-nonconforming virtual element as
\begin{align*}
\mathbb{V}_{k+1}^{\rm VE}(T):=\{\bs v\in H^1(T;\mathbb R^d): &\div\bs v\in\mathbb P_{k}(T), \textrm{ there exists some } s\in L^2(T) \\
&  \textrm{ such that } \Delta\bs v+\nabla s\in \big(\mathbb P_{k-1}(T;\mathbb R^d)\cap\ker(\cdot\boldsymbol x)\big),  \\
&\textrm{ and } (\partial_n\bs v+s\bs n)|_F\in\mathbb P_{k}(F;\mathbb R^d) \;\forall~F\in\mathcal F(T)\}.
\end{align*}
Following the argument in \cite[Section 3.1]{WeiHuangLi2021},
the local virtual element space $\mathbb{V}_{k+1}^{\rm VE}(T)$ is uniquely determined by the following DoFs
\begin{subequations}\label{vev-dof}
\begin{align}
(\bs v, \bs q)_F, & \quad \bs q\in\mathbb P_{k}(F; \mathbb R^d),  F\in\mathcal F(T), \label{vev-dof1}\\
(\bs v, \bs q)_T, & \quad \bs q\in\mathbb P_{k-1}(T;\mathbb R^d). \label{vev-dof2}
\end{align}
\end{subequations}
Clearly we have $\mathbb P_{k+1}(T;\mathbb R^d)\subseteq\mathbb{V}_{k+1}^{\rm VE}(T)$, and $\mathbb{V}_{1}^{\rm VE}(T)=\mathbb P_1(T;\mathbb R^d)$.
Then define the global virtual element space by
\begin{align*}
\mathring{\mathbb{V}}_{h}^{\rm VE} &:= \{\boldsymbol{v}\in L^2(\Omega;\mathbb R^d):\boldsymbol{v}|_T\in\mathbb{V}_{k+1}^{\rm VE}(T) \textrm{ for each }T\in\mathcal{T}_h; \textrm{ DoF \eqref{vev-dof1} is } \\
&\qquad\qquad\qquad\quad\quad\;\;\textrm{ single-valued across each face in $\mathring{\mathcal{F}}_h$, and vanishes on $\partial\Omega$}\}.
\end{align*}
A mixed virtual element method for the Stokes equation \eqref{eq:stokes} is to find $\boldsymbol{u}_h\in \mathring{\mathbb{V}}_{h}^{\rm VE}$ and $p_h\in\mathbb{P}_{k}(\mathcal{T}_h)/\mathbb R$ with integer $k\geq 0$, such that
\begin{subequations}\label{distribustokesvem}
\begin{align}
\label{distribustokesvem1}
(Q_{k}(\dev \grad_h\boldsymbol{u}_h), Q_{k}(\dev \grad_h\boldsymbol{v}_h)) + (\div_h\boldsymbol{v}_h, p_h) &= (\boldsymbol{f}, I_{k,k}^{\div}\boldsymbol{v}_h), \\
\label{distribustokesvem2}
(\div_h\boldsymbol{u}_h, q_h)&=0
\end{align}
\end{subequations}
for all $\boldsymbol{v}_h\in \mathring{\mathbb{V}}_{h}^{\rm VE}$ and $q_h\in\mathbb{P}_{k}(\mathcal{T}_h)/\mathbb R$.
When $k=0$, the mixed method \eqref{distribustokesvem} is exactly the modified Crouzeix-Raviart element method in \cite[(23)]{Linke2014}.


\begin{lemma}
It holds
\begin{equation}\label{eq:weakdevgradvemcommu}
\dev \grad_w(I_{k,k}^{\div}\boldsymbol{v}_h, Q_{k,\mathcal{F}_h}(\Pi_F\boldsymbol{v}_h))	= Q_{k}(\dev \grad_h\boldsymbol{v}_h)\quad\forall\,\boldsymbol{v}_h\in \mathring{\mathbb{V}}_{h}^{\rm VE}.
\end{equation}
\end{lemma}
\begin{proof}
By the definitions of the operator $\dev \grad_w$ and the operator $I_{k,k}^{\div}$,	
we have for any $\boldsymbol{\tau}\in \mathbb{P}_{k}(T;\mathbb{T})$ and $T\in\mathcal{T}_h$ that
\begin{align*}
&\quad\; (\dev \grad_w(I_{k,k}^{\div}\boldsymbol{v}_h, Q_{k,\mathcal{F}_h}(\Pi_F\boldsymbol{v}_h)), \boldsymbol{\tau})_T \\
&=-(I_{k,k}^{\div}\boldsymbol{v}_h,\div\boldsymbol{\tau})_T+(\boldsymbol{n}\cdot(I_{k,k}^{\div}\boldsymbol{v}_h),\boldsymbol{n}^{\intercal}\boldsymbol{\tau}\boldsymbol{n})_{\partial T}+(\Pi_F\boldsymbol{v}_h, \Pi_F\boldsymbol{\tau}\boldsymbol{n})_{\partial T}\\
&=-(\boldsymbol{v}_h,\div\boldsymbol{\tau})_T+(\boldsymbol{n}\cdot\boldsymbol{v}_h,\boldsymbol{n}^{\intercal}\boldsymbol{\tau}\boldsymbol{n})_{\partial T}+(\Pi_F\boldsymbol{v}_h, \Pi_F\boldsymbol{\tau}\boldsymbol{n})_{\partial T}.
\end{align*}
Then \eqref{eq:weakdevgradvemcommu} follows from the integration by parts.
\end{proof}

As a consequence of \eqref{eq:weakdevgradvemcommu}, $\|Q_{k}(\dev \grad_h\boldsymbol{v}_h)\|$ defines a norm on the space $\mathring{\mathbb{V}}_{h}^{\rm VE}\cap\ker(\div_h)$.

\begin{theorem}
The mixed virtual element method \eqref{distribustokesvem} is well-posed.
Let $(\boldsymbol{u}_h, p_h)\in \mathring{\mathbb{V}}_{h}^{\rm VE} \times \mathbb{P}_{k}(\mathcal{T}_h)/\mathbb R$ be its solution.
Then $(I_{k,k}^{\div}\boldsymbol{u}_h, Q_{k,\mathcal{F}_h}(\Pi_F\boldsymbol{u}_h), Q_{\ell}p_h)\in \mathring{\mathbb{V}}_{k,\ell}^{\div} \times\mathbb P_{k}(\mathring{\mathcal{F}}_h$; $\mathbb R^{d-1})$ $\times \mathbb{P}_{\ell}(\mathcal{T}_h)/\mathbb R$ be the solution of the mixed finite element method \eqref{distribustokesfemWG}.
\end{theorem}
\begin{proof}
The well-posedness of the mixed finite element method \eqref{distribustokesvem} can be proved using an argument analogous to that in Theorem~\ref{thm:equivalence0}.

We then prove the equivalence between the discrete method \eqref{distribustokesvem} and the discrete method \eqref{distribustokesfemWG}.  By \eqref{eq:weakdevgradvemcommu}, the equation \eqref{distribustokesvem1} means that $(I_{k,k}^{\div}\boldsymbol{u}_h, Q_{k,\mathcal{F}_h}(\Pi_F\boldsymbol{u}_h), Q_{\ell}p_h)$ satisfies equation \eqref{distribustokesfemWG1}. By the commuting property \eqref{eq:Ihdivprop1}, the equation \eqref{distribustokesvem2} indicates that $I_{k,k}^{\div}\boldsymbol{u}_h$ satisfies equation \eqref{distribustokesfemWG2}.
\end{proof}

\begin{remark}\rm
We can also use the following virtual element space for the velocity based on \eqref{ncve}:
\begin{align*}
&\{\boldsymbol{v}\in L^2(\Omega;\mathbb R^d):\boldsymbol{v}|_T\in V_{k+1,k+1}^{\rm VE}(T)\otimes\mathbb R^d \textrm{ for each }T\in\mathcal{T}_h; \textrm{ DoF \eqref{ve-dof1} is } \\
&\qquad\qquad\qquad\quad\quad\;\;\textrm{ single-valued across each face in $\mathring{\mathcal{F}}_h$, and vanishes on $\partial\Omega$}\}.
\end{align*}
\end{remark}

\subsection{Mass-Conserving Mixed Stress Method}
By introducing $\boldsymbol{\sigma}_h:= \dev \grad_w\boldsymbol{u}_h\in\Sigma_{k}^{\rm tn}$, we can rewrite the mixed finite element method \eqref{distribustokesfemWG} as the MCS method~\cite{GopalakrishnanLedererSchoeberl2020a}: find $\boldsymbol{\sigma}_h\in \Sigma_{k}^{\rm tn}$, $\boldsymbol{u}_h\in\mathring{\mathbb{V}}_{k,\ell}^{\div}$ and $p_h\in\mathbb{P}_{\ell}(\mathcal{T}_h)/\mathbb R$ such that
\begin{align*}
(\boldsymbol{\sigma}_h,\boldsymbol{\tau}_h)+b_h(\boldsymbol{\tau}_h,q_h;\boldsymbol{u}_h)&=0  & & \forall~\boldsymbol{\tau}_h \in \Sigma_{k}^{\rm tn}, q_h\in\mathbb{P}_{\ell}(\mathcal{T}_h)/\mathbb R, \\
b_h(\boldsymbol{\sigma}_h,p_h;\boldsymbol{v}_h) &=-(\boldsymbol{f}, \boldsymbol{v}_h) & &\forall~\boldsymbol{v}_h \in \mathring{\mathbb{V}}_{k,\ell}^{\div},
\end{align*}
where the bilinear form
\begin{align*}
b_h(\boldsymbol{\tau}_h,q_h;\boldsymbol{v}_h)&:=\sum_{T\in \mathcal{T}_{h}}(\div\boldsymbol{\tau}_h, \boldsymbol{v}_h)_T-\sum_{T\in \mathcal{T}_{h}}(\boldsymbol{n}^\intercal\boldsymbol{\tau}_h\boldsymbol{n}, \boldsymbol{n}\cdot\boldsymbol{v}_h)_{\partial T} - (\div\boldsymbol{v}_h, q_h) \\
&\;=-\sum_{T\in \mathcal{T}_{h}}(\boldsymbol{\tau}_h, \grad\boldsymbol{v}_h)_T+\sum_{T\in \mathcal{T}_{h}}(\Pi_F\boldsymbol{\tau}_h\boldsymbol{n}, \Pi_F\boldsymbol{v}_h)_{\partial T} - (\div\boldsymbol{v}_h, q_h).
\end{align*}

The subspace
\begin{equation*}
\left\{
\boldsymbol{\tau}_h \in \Sigma_{k}^{\rm tn} :
\Pi_F \boldsymbol{\tau}_h \boldsymbol{n} \in \mathbb{P}_{k-1}(F; \mathbb{R}^{d-1})
\ \text{for all } F \in \mathcal{F}_h
\right\}
\end{equation*}
was employed in the MCS method~\cite{GopalakrishnanLedererSchoeberl2020a} 
to discretize the traceless stress tensor $\boldsymbol{\sigma}$.  
However, the constraint 
$\Pi_F \boldsymbol{\tau}_h \boldsymbol{n} \in \mathbb{P}_{k-1}(F; \mathbb{R}^{d-1})$
prevents the decoupling of the error and thus not achieving superconvergence properties such as 
the estimates 
$\|\bs \sigma - \bs \sigma_h\|$ 
and $\|Q_{\ell} p - p_h\|$ in~\eqref{eq:errorestimate1}, 
and $\|I_{k,k}^{\div}\boldsymbol{u} - \boldsymbol{u}_h\|$ in~\eqref{eq:errorestimate4}.  

Superconvergence was later obtained in~\cite{GopalakrishnanLedererSchoeberl2020} by enriching the stress space. That approach requires $k=\ell\ge 1$, whereas our formulation also covers the low-order cases $(k,\ell)=(0,-1)$, $(0,0)$, and $(1,0)$. Our analysis seems simpler and more transparent.

\section{Numerical Examples}\label{sec:numericalresults}
In this section, we verify the error estimates for the mixed finite element method~\eqref{distribustokesfemWG} and the postprocessing scheme~\eqref{distribustokespostfem}. Both smooth and singular test problems are considered, with focus on the low-order cases $\ell=0$ and $k=0,1$. The smooth benchmark problems are adapted from~\cite{GopalakrishnanLedererSchoeberl2020}, where examples for $k=\ell \geq 1$ are presented. Throughout, the computational domains are discretized by uniform simplicial meshes, and all computations are carried out in MATLAB using $i$FEM~\cite{Chen2009}.

\subsection{Smooth problems with non-slip boundary conditions}

\begin{example}\label{exm1}
	\normalfont
	We first test the mixed finite element method~\eqref{distribustokesfemWG} in two dimensions on the unit square domain $\Omega=(0,1)^2$. The exact solution is
	$$
	\boldsymbol{u} = \operatorname{curl}\psi_2,
	\qquad
	p = -x^5 - y^5 + \frac{1}{3},
	$$
	where $\psi_2 = x^2(x - 1)^2 y^2(y - 1)^2$.
\end{example}

\begin{table}[htbp]
	\centering
	\caption{Errors for Example~\ref{exm1} with $\ell=0$ and $k=0,1$ in two dimensions.}
	\label{table:numericalerr2Dmerged}
	\renewcommand{\arraystretch}{1.125}
	\resizebox{\textwidth}{!}{
		\begin{tabular}{c|c|cc|cc|cc|cc}
			\toprule
			$h$ & $(k,\ell)$
			& $\|\boldsymbol{\sigma}-\boldsymbol{\sigma}_h\|_{0,h}$ & order
			& $\|p-p_h\|$ & order
			& $\|\boldsymbol{u} - \boldsymbol{u}_h^{*}\|$ & order
			& $\|\nabla_h(\boldsymbol{u} - \boldsymbol{u}_h^{*})\|$ & order \\
			\midrule
			$2^{-3}$ & $(0,0)$ & 3.195e-02 & -- & 7.473e-02 & -- & 1.233e-03 & -- & 2.890e-02 & -- \\
			$2^{-4}$ & $(0,0)$ & 1.597e-02 & 1.00 & 3.768e-02 & 0.99 & 3.277e-04 & 1.91 & 1.481e-02 & 0.96 \\
			$2^{-5}$ & $(0,0)$ & 7.973e-03 & 1.00 & 1.887e-02 & 1.00 & 8.353e-05 & 1.97 & 7.453e-03 & 0.99 \\
			$2^{-6}$ & $(0,0)$ & 3.982e-03 & 1.00 & 9.442e-03 & 1.00 & 2.099e-05 & 1.99 & 3.733e-03 & 1.00 \\
			$2^{-7}$ & $(0,0)$ & 1.990e-03 & 1.00 & 4.721e-03 & 1.00 & 5.256e-06 & 2.00 & 1.867e-03 & 1.00 \\
			\midrule
			$2^{-3}$ & $(1,0)$ & 2.966e-03 & -- & 7.453e-02 & -- & 3.296e-05 & -- & 2.286e-03 & -- \\
			$2^{-4}$ & $(1,0)$ & 7.317e-04 & 2.02 & 3.760e-02 & 0.99 & 4.167e-06 & 2.98 & 5.813e-04 & 1.98 \\
			$2^{-5}$ & $(1,0)$ & 1.806e-04 & 2.02 & 1.884e-02 & 1.00 & 5.264e-07 & 2.99 & 1.463e-04 & 1.99 \\
			$2^{-6}$ & $(1,0)$ & 4.480e-05 & 2.01 & 9.428e-03 & 1.00 & 6.625e-08 & 2.99 & 3.666e-05 & 2.00 \\
			$2^{-7}$ & $(1,0)$ & 1.115e-05 & 2.01 & 4.715e-03 & 1.00 & 8.315e-09 & 2.99 & 9.178e-06 & 2.00 \\
			\bottomrule
		\end{tabular}
	}
\end{table}

\begin{example}\label{exm2}
	\normalfont
	We next consider the mixed finite element method~\eqref{distribustokesfemWG} in three dimensions on the unit cube domain $\Omega=(0,1)^3$. The exact solution is given by
	\[
	\boldsymbol{u} = \operatorname{curl}(\psi_3, \psi_3, \psi_3)^{\intercal},
	\qquad
	p = - x^5 - y^5 - z^5 + \frac{1}{2},
	\]
	where $\psi_3 = x^2(x - 1)^2 y^2(y - 1)^2 z^2(z - 1)^2$.
\end{example}

\begin{table}[htbp]
	\centering
	\caption{Errors for Example~\ref{exm2} with $\ell=0$ and $k=0,1$ in three dimensions.}
	\label{table:numericalerr3Dmerged}
	\renewcommand{\arraystretch}{1.125}
	\resizebox{\textwidth}{!}{
		\begin{tabular}{c|c|cc|cc|cc|cc}
			\toprule
			$h$ & $(k,\ell)$
			& $\|\boldsymbol{\sigma}-\boldsymbol{\sigma}_h\|_{0,h}$ & order
			& $\|p-p_h\|$ & order
			& $\|\boldsymbol{u} - \boldsymbol{u}_h^{*}\|$ & order
			& $\|\nabla_h(\boldsymbol{u} - \boldsymbol{u}_h^{*})\|$ & order \\
			\midrule
			$2^{-1}$ & $(0,0)$ & 9.933e-03 & -- & 2.942e-01 & -- & 4.167e-04 & -- & 4.576e-03 & -- \\
			$2^{-2}$ & $(0,0)$ & 4.718e-03 & 1.07 & 1.649e-01 & 0.84 & 1.565e-04 & 1.41 & 2.880e-03 & 0.67 \\
			$2^{-3}$ & $(0,0)$ & 2.301e-03 & 1.04 & 8.501e-02 & 0.96 & 4.400e-05 & 1.83 & 1.539e-03 & 0.90 \\
			$2^{-4}$ & $(0,0)$ & 1.135e-03 & 1.02 & 4.284e-02 & 0.99 & 1.143e-05 & 1.94 & 7.836e-04 & 0.97 \\
			$2^{-5}$ & $(0,0)$ & 5.638e-04 & 1.01 & 2.146e-02 & 1.00 & 2.889e-06 & 1.98 & 3.937e-04 & 0.99 \\
			\midrule
			$2^{-1}$ & $(1,0)$ & 4.048e-03 & -- & 2.942e-01 & -- & 9.624e-05 & -- & 1.894e-03 & -- \\
			$2^{-2}$ & $(1,0)$ & 1.045e-03 & 1.95 & 1.649e-01 & 0.84 & 1.493e-05 & 2.69 & 5.478e-04 & 1.79 \\
			$2^{-3}$ & $(1,0)$ & 2.631e-04 & 1.99 & 8.501e-02 & 0.96 & 2.035e-06 & 2.88 & 1.486e-04 & 1.88 \\
			$2^{-4}$ & $(1,0)$ & 6.465e-05 & 2.02 & 4.284e-02 & 0.99 & 2.627e-07 & 2.95 & 3.808e-05 & 1.96 \\
			\bottomrule
		\end{tabular}
	}
\end{table}

		Numerical errors for the smooth problems (Examples~\ref{exm1} and~\ref{exm2}) are reported in Tables~\ref{table:numericalerr2Dmerged} and~\ref{table:numericalerr3Dmerged}.
		The observed rates agree with the theoretical estimates~\eqref{eq:errorestimate1} and~\eqref{eq:postH1error}--\eqref{eq:postL2error}. In particular,
	$$
	\|\boldsymbol{\sigma}-\boldsymbol{\sigma}_h\|_{0,h}
	= \mathcal{O}(h^{k+1}),
	\qquad
	\|p-p_h\| = \mathcal{O}(h),
	$$
	and the postprocessed velocity satisfies
	$$
	\|\boldsymbol{u} - \boldsymbol{u}_h^*\|
	= \mathcal{O}(h^{k+2}),
	\qquad
	\|\nabla_h(\boldsymbol{u}-\boldsymbol{u}_h^*)\|
	= \mathcal{O}(h^{k+1}).
	$$
%

\subsection{Singular problem with slip boundary conditions}

We next consider a singular test problem on the two-dimensional L-shaped domain
$\Omega=(-1,1)^2\setminus\big([0,1]\times[-1,0]\big)$ with boundary condition
\[
\boldsymbol{u}\cdot\boldsymbol{n}=0 \quad\text{and}\quad \boldsymbol{t}^{\intercal}\boldsymbol{\sigma}\boldsymbol{n}=0 \quad\text{on }\partial\Omega.
\]
The re-entrant corner at the origin induces reduced regularity of the exact solution and hence reduced convergence rates on uniform meshes.

We impose this boundary condition by using the space $\mathring{\Sigma}_{k}^{\rm tn}$. See Remark~\ref{rm:infsupsigma0}.

\begin{example}\label{ex:Lshape-slip-hom}
		\normalfont
Let $(r,\theta)$ denote polar coordinates centered at the re-entrant corner of $\Omega$.
The corner is located at $(0,0)$, and $\theta\in(0,3\pi/2)$ in $\Omega$. Define, 
		\[
		\Phi(x,y)=(1-x^2)^3(1-y^2)^3,
		\qquad
		\Psi_s(r,\theta)=r^{1+\lambda}\sin((\lambda-1)\theta),
		\qquad
		\lambda=\tfrac13,
		\]
		and set
		\[
		\boldsymbol{u}
		= \curl (\Phi\Psi_s).
		\]
		Further, let 
		\[
		g(\theta)=\sin((\lambda-1)\theta), 
		\quad
		Q(\theta)=(1+\lambda)^2g(\theta)+g''(\theta),
		\quad
		p_s(r,\theta)=\frac{1}{\lambda-1}r^{\lambda-1}Q'(\theta).
		\]
		The pressure is defined by
		\[
		p(x,y)=\Phi(x,y)p_s(r,\theta)-c_0,
		\quad
		c_0=|\Omega|^{-1}\int_\Omega \Phi p_s\,dx,
		\]
		so that $p\in L_0^2(\Omega)$.
\end{example}

\begin{figure}[htbp]
	\centering
	\captionsetup[subfigure]{skip=2pt}
	\setlength{\tabcolsep}{2pt}
	
	\begin{tabular}{ccc}
		\begin{subfigure}{0.32\textwidth}
			\centering
			\includegraphics[width=0.9\linewidth]{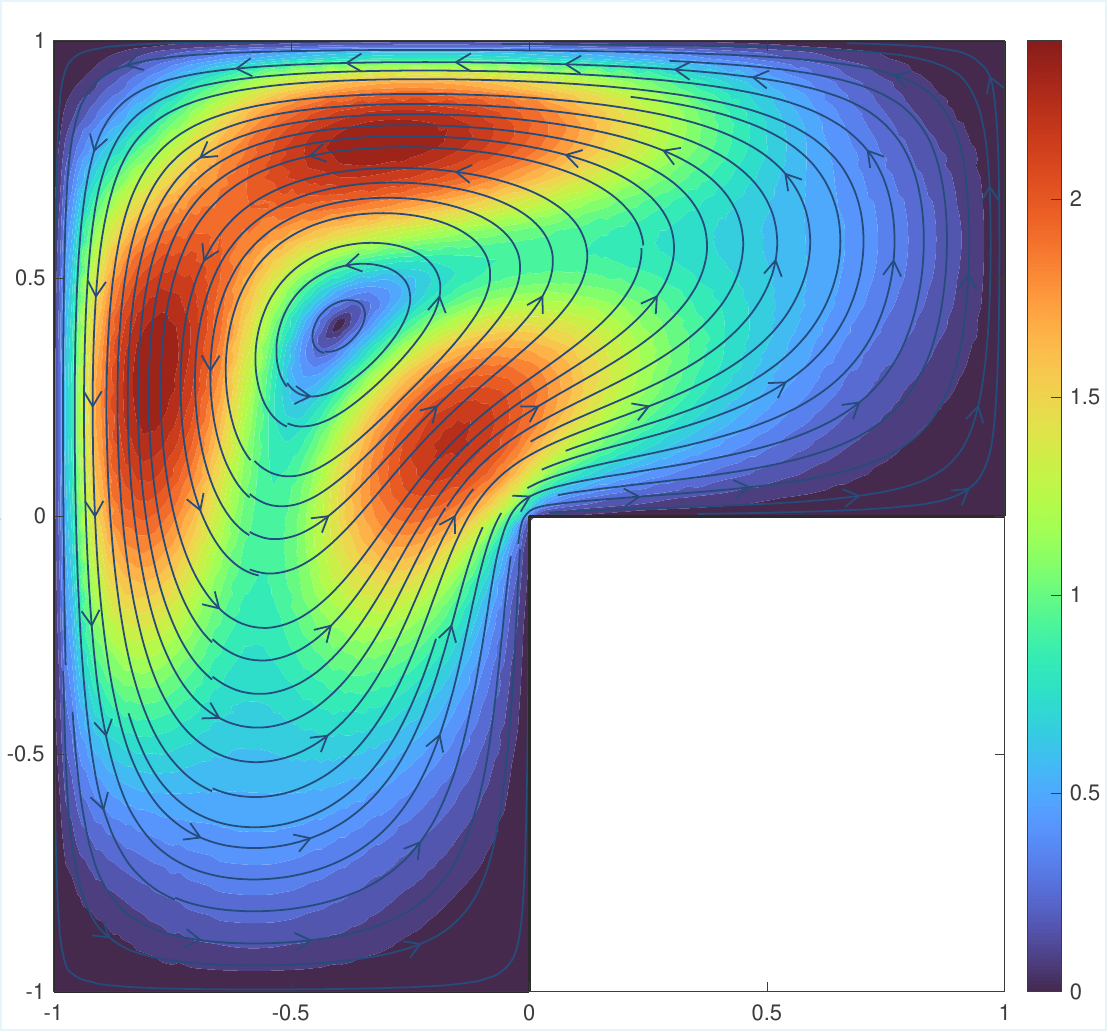}
			\caption{$\boldsymbol{u}_h \text{ for } k=0, \ell=0$}
		\end{subfigure}
		\hfill 
		\begin{subfigure}{0.32\textwidth}
			\centering
			\includegraphics[width=0.9\linewidth]{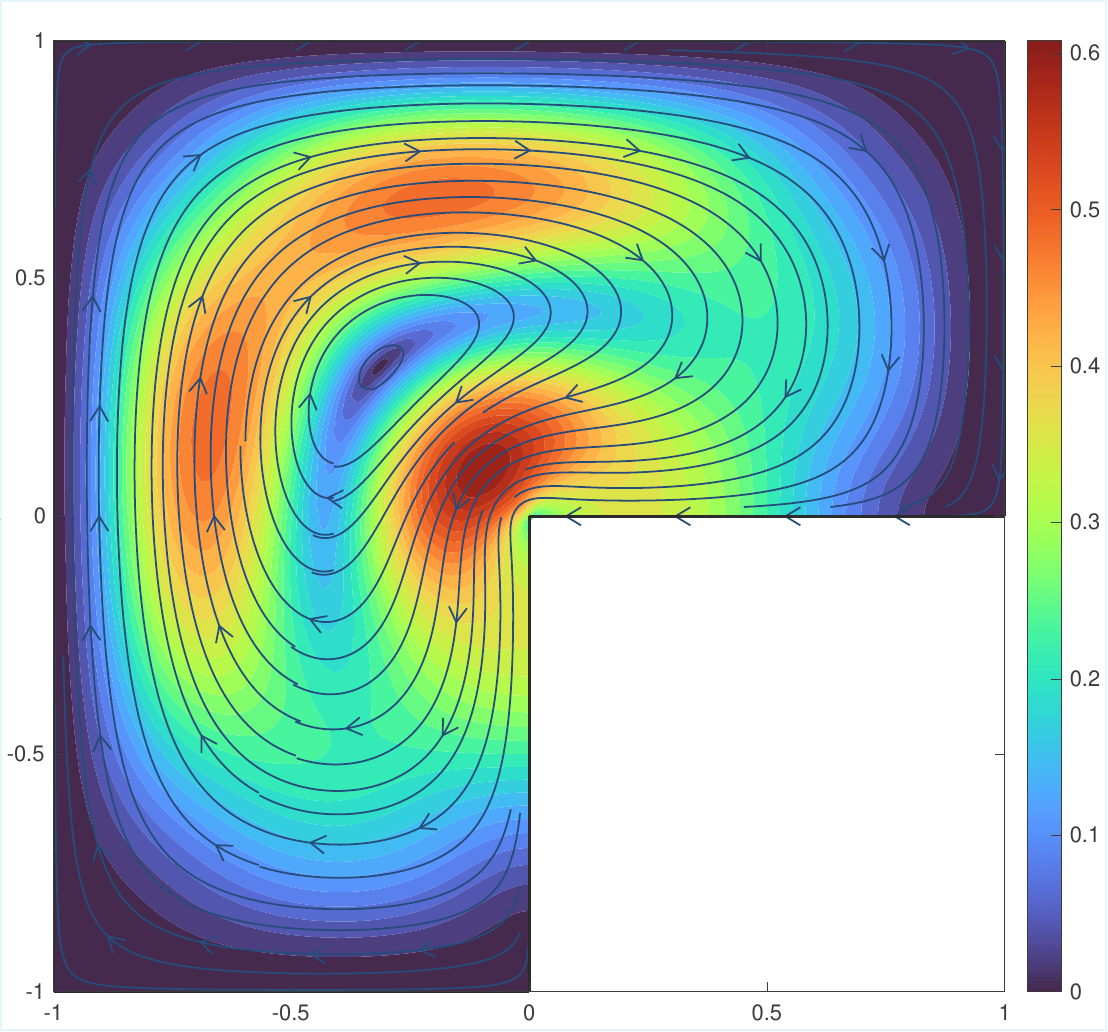}
			\caption{$\boldsymbol{u}_h \text{ for } k=1, \ell=0$}
		\end{subfigure}
		\hfill 
		\begin{subfigure}{0.32\textwidth}
			\centering
			\includegraphics[width=0.9\linewidth]{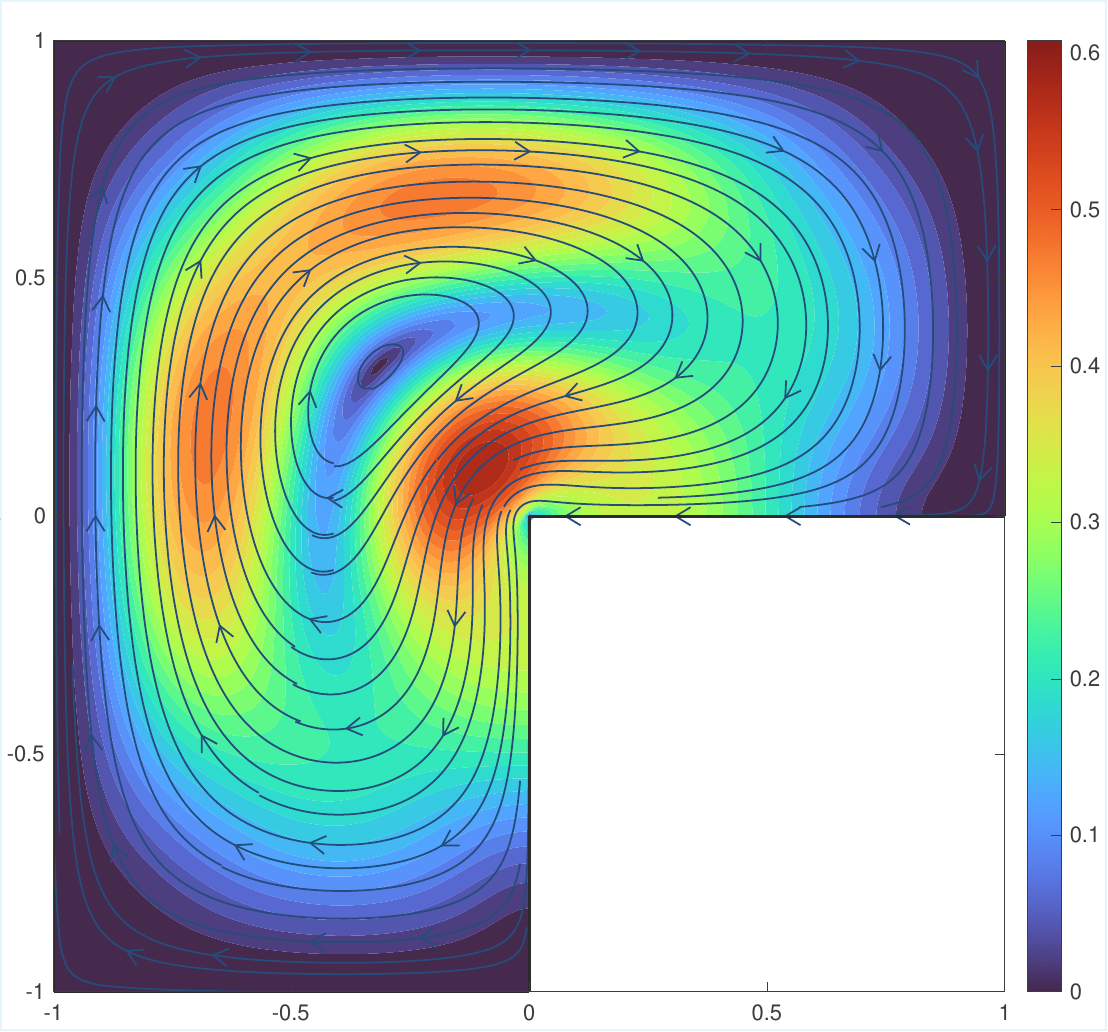}
			\caption{Exact solution $\boldsymbol{u}$}
		\end{subfigure}
	\end{tabular}
\caption{Flow structure of the velocity field $\boldsymbol{u}$ and its approximation $\boldsymbol{u}_h$ on the L-shaped domain for mesh size $h=2^{-7}$. Left and middle: numerical solutions. Right: exact solution.}
	\label{fig:Velocity_field}
\end{figure}

\begin{table}[htbp]
	\centering
		\caption{Errors for Example~\ref{ex:Lshape-slip-hom} with $\ell=0$ and $k=0,1$ in the L-shaped domain.}
		\label{table:numericalerr2DLRT0BDM1}
		\renewcommand{\arraystretch}{1.125}
		\resizebox{10.25cm}{!}{
		\begin{tabular}{c|c|cc|cc|cc}
			\toprule
			$h$ & $(k,\ell)$
			& $\|\boldsymbol{u}-\boldsymbol{u}_h\|$ & order
			& $\|\boldsymbol{\sigma}-\boldsymbol{\sigma}_h\|_{0,h}$ & order
			& $\|p-p_h\|$ & order \\
			\midrule
			$2^{-3}$  & $(0,0)$ & 1.867e-01 & -- & 2.572e+00 & -- & 1.364e+00 & -- \\
			$2^{-4}$  & $(0,0)$ & 1.156e-01 & 0.69 & 1.928e+00 & 0.42 & 1.006e+00 & 0.44 \\
			$2^{-5}$  & $(0,0)$ & 7.489e-02 & 0.63 & 1.490e+00 & 0.37 & 7.250e-01 & 0.47 \\
			$2^{-6}$ & $(0,0)$ & 4.875e-02 & 0.62 & 1.168e+00 & 0.35 & 5.170e-01 & 0.49 \\
			$2^{-7}$ & $(0,0)$ & 3.153e-02 & 0.63 & 9.217e-01 & 0.34 & 3.693e-01 & 0.49 \\
			\midrule
			$2^{-3}$  & $(1,0)$ & 7.878e-02 & -- & 1.366e+00 & -- & 8.350e-01 & -- \\
			$2^{-4}$  & $(1,0)$ & 5.256e-02 & 0.58 & 1.070e+00 & 0.35 & 6.071e-01 & 0.46 \\
			$2^{-5}$  & $(1,0)$ & 3.457e-02 & 0.60 & 8.453e-01 & 0.34 & 4.423e-01 & 0.46 \\
			$2^{-6}$ & $(1,0)$ & 2.242e-02 & 0.62 & 6.686e-01 & 0.34 & 3.257e-01 & 0.44 \\
			$2^{-7}$ & $(1,0)$ & 1.440e-02 & 0.64 & 5.294e-01 & 0.34 & 2.434e-01 & 0.42 \\
			\bottomrule
		\end{tabular}
	}
\end{table}

The numerical results for Example~\ref{ex:Lshape-slip-hom} are reported in Table~\ref{table:numericalerr2DLRT0BDM1}. The corner singularity implies that $\boldsymbol{\sigma},p\in H^{1/3-\varepsilon}(\Omega)$. Hence, on uniform meshes, one expects
$$
\|\boldsymbol{\sigma}-\boldsymbol{\sigma}_h\|_{0,h}=\mathcal{O}(h^{1/3}),
\quad
\|p-p_h\|=\mathcal{O}(h^{1/3}),
\quad
\|\boldsymbol{u}-\boldsymbol{u}_h\|=\mathcal{O}(h^{2/3}).
$$
The observed rates are consistent with these predictions. Fig.~\ref{fig:Velocity_field} shows the velocity field on the L-shaped domain for two numerical solutions and the exact solution. In particular, the effect of the re-entrant corner is clearly visible.  Although the convergence order of $\boldsymbol{u}_h$ is limited by the regularity, BDM$_1$ with $(k,\ell)=(1,0)$ still performs better than RT$_0$ with $(k,\ell)=(0,0)$ away from the singularity.

\section{Conclusions and Future Work}\label{sec:conclusion}
In this paper, we develop a family of pointwise divergence-free mixed finite element methods for the Stokes equation. The method relies on a distributional discretization of the vector Laplacian, which gives a consistent and stable scheme even though the velocity space is only $H(\div)$-conforming. The key ingredient is the weak divergence operator $\div_w$, together with a distributional divergence-free property of the discrete stress against divergence-free velocity fields. This structure decouples the stress and velocity errors, yields velocity supercloseness together with an optimal projected-pressure estimate when $\ell=k$ and a superconvergent projected-pressure estimate when $\ell=k-1$, and leads to a superconvergent postprocessed velocity.

The framework also suggests several directions for future work:
\begin{enumerate}[leftmargin = 16pt]
    \item \emph{Brinkman and Navier--Stokes equations}: The same framework is natural for the Brinkman model, where it captures the transition between Darcy and Stokes regimes. For high-Reynolds-number Navier--Stokes flows, additional stabilization is still needed when convection dominates.
    
    \item \emph{Linear elasticity}: The link between stress variables and divergence constraints suggests extensions to mixed elasticity. Symmetric tensor spaces in this setting may lead to robust elements for nearly incompressible materials.
    
    \item \emph{Boundary conditions and regularity}: The weak operators can be adapted to slip and traction boundary conditions through modified boundary trace terms. It is also important to study superconvergence and adaptive finite element methods on nonconvex domains.
\end{enumerate}

\bibliographystyle{abbrv}
\bibliography{./refs}

\begin{thebibliography}{10}

\bibitem{Arnold2018}
D.~N. Arnold.
\newblock {\em Finite element exterior calculus}, volume~93 of {\em CBMS-NSF Regional Conference Series in Applied Mathematics}.
\newblock Society for Industrial and Applied Mathematics (SIAM), Philadelphia, PA, 2018.

\bibitem{Arnold.D;Brezzi.F1985}
D.~N. Arnold and F.~Brezzi.
\newblock Mixed and nonconforming finite element methods: {{Implementation}}, postprocessing and error estimates.
\newblock {\em RAIRO Model Math. Anal. Numer.}, 19:7--32, 1985.

\bibitem{ArnoldBrezziFortin1984}
D.~N. Arnold, F.~Brezzi, and M.~Fortin.
\newblock A stable finite element for the {S}tokes equations.
\newblock {\em Calcolo}, 21(4):337--344, 1984.

\bibitem{ArnoldFalkGopalakrishnan2012}
D.~N. Arnold, R.~S. Falk, and J.~Gopalakrishnan.
\newblock Mixed finite element approximation of the vector {L}aplacian with {D}irichlet boundary conditions.
\newblock {\em Math. Models Methods Appl. Sci.}, 22(9):1250024, 26, 2012.

\bibitem{ArnoldFalkWinther2006}
D.~N. Arnold, R.~S. Falk, and R.~Winther.
\newblock Finite element exterior calculus, homological techniques, and applications.
\newblock {\em Acta Numer.}, 15:1--155, 2006.

\bibitem{ArnoldHu2021}
D.~N. Arnold and K.~Hu.
\newblock Complexes from complexes.
\newblock {\em Found. Comput. Math.}, 21(6):1739--1774, 2021.

\bibitem{ArnoldQin1992a}
D.~N. Arnold and J.~Qin.
\newblock Quadratic velocity/linear pressure {S}tokes elements.
\newblock In {\em Advances in Computer Methods for Partial Differential Equations-VII}, IMACS, pages 28---34, 1992.

\bibitem{AyusodeDiosLipnikovManzini2016}
B.~Ayuso~de Dios, K.~Lipnikov, and G.~Manzini.
\newblock The nonconforming virtual element method.
\newblock {\em ESAIM Math. Model. Numer. Anal.}, 50(3):879--904, 2016.

\bibitem{BeiraodaVeigaLovadinaVacca2017}
L.~Beir\~ao~da Veiga, C.~Lovadina, and G.~Vacca.
\newblock Divergence free virtual elements for the {S}tokes problem on polygonal meshes.
\newblock {\em ESAIM Math. Model. Numer. Anal.}, 51(2):509--535, 2017.

\bibitem{BoffiBrezziFortin2013}
D.~Boffi, F.~Brezzi, and M.~Fortin.
\newblock {\em Mixed finite element methods and applications}, volume~44 of {\em Springer Series in Computational Mathematics}.
\newblock Springer, Heidelberg, 2013.

\bibitem{Braess.D;Schoberl.J2008}
D.~Braess and J.~Sch\"oberl.
\newblock Equilibrated residual error estimator for edge elements.
\newblock {\em Math. Comp.}, 77(262):651--672, 2008.

\bibitem{Brezzi1974}
F.~Brezzi.
\newblock On the existence, uniqueness and approximation of saddle-point problems arising from {L}agrangian multipliers.
\newblock {\em Rev. Fran\c caise Automat. Informat. Recherche Op\'erationnelle S\'er. Rouge}, 8:129--151, 1974.

\bibitem{BrezziDouglasMarini1985}
F.~Brezzi, J.~Douglas, Jr., and L.~D. Marini.
\newblock Two families of mixed finite elements for second order elliptic problems.
\newblock {\em Numer. Math.}, 47(2):217--235, 1985.

\bibitem{ChenHuangWei2024}
C.~Chen, X.~Huang, and H.~Wei.
\newblock Virtual element methods without extrinsic stabilization.
\newblock {\em SIAM J. Numer. Anal.}, 62(1):567--591, 2024.

\bibitem{ChenFengXie2016}
G.~Chen, M.~Feng, and X.~Xie.
\newblock Robust globally divergence-free weak {G}alerkin methods for {S}tokes equations.
\newblock {\em J. Comput. Math.}, 34(5):549--572, 2016.

\bibitem{Chen2009}
L.~Chen.
\newblock $i${FEM}: an integrated finite element methods package in {MATLAB}.
\newblock Technical report, University of California at Irvine, 2009.

\bibitem{ChenHuang2020}
L.~Chen and X.~Huang.
\newblock Nonconforming virtual element method for {$2m$}th order partial differential equations in {$\Bbb{R}^n$}.
\newblock {\em Math. Comp.}, 89(324):1711--1744, 2020.

\bibitem{ChenHuang2022}
L.~Chen and X.~Huang.
\newblock Finite elements for div- and divdiv-conforming symmetric tensors in arbitrary dimension.
\newblock {\em SIAM J. Numer. Anal.}, 60(4):1932--1961, 2022.

\bibitem{ChenHuang2024a}
L.~Chen and X.~Huang.
\newblock Finite element de {R}ham and {S}tokes complexes in three dimensions.
\newblock {\em Math. Comp.}, 93(345):55--110, 2024.

\bibitem{ChenHuang2025}
L.~Chen and X.~Huang.
\newblock Finite element complexes in two dimensions (in {C}hinese).
\newblock {\em Sci. Sin. Math.}, 55(8):1593--1626, 2025.

\bibitem{ChenHuangZhang2025}
L.~Chen, X.~Huang, and C.~Zhang.
\newblock Distributional finite element curl div complexes and application to quad curl problems.
\newblock {\em SIAM J. Numer. Anal.}, 63(3):1078--1104, 2025.

\bibitem{ChenWang2019}
L.~Chen and F.~Wang.
\newblock A divergence free weak virtual element method for the {S}tokes problem on polytopal meshes.
\newblock {\em J. Sci. Comput.}, 78(2):864--886, 2019.

\bibitem{ChenWangZhong2015}
L.~Chen, M.~Wang, and L.~Zhong.
\newblock Convergence analysis of triangular {MAC} schemes for two dimensional {S}tokes equations.
\newblock {\em J. Sci. Comput.}, 63(3):716--744, 2015.

\bibitem{ChristiansenHu2018}
S.~H. Christiansen and K.~Hu.
\newblock Generalized finite element systems for smooth differential forms and {S}tokes' problem.
\newblock {\em Numer. Math.}, 140(2):327--371, 2018.

\bibitem{CockburnSayas2014}
B.~Cockburn and F.-J. Sayas.
\newblock Divergence-conforming {HDG} methods for {S}tokes flows.
\newblock {\em Math. Comp.}, 83(288):1571--1598, 2014.

\bibitem{CostabelMcIntosh2010}
M.~Costabel and A.~McIntosh.
\newblock On {B}ogovski\u{\i} and regularized {P}oincar\'{e} integral operators for de {R}ham complexes on {L}ipschitz domains.
\newblock {\em Math. Z.}, 265(2):297--320, 2010.

\bibitem{CrouzeixRaviart1973}
M.~Crouzeix and P.-A. Raviart.
\newblock Conforming and nonconforming finite element methods for solving the stationary {S}tokes equations. {I}.
\newblock {\em Rev. Fran\c{c}aise Automat. Informat. Recherche Op\'{e}rationnelle S\'{e}r. Rouge}, 7:33--75, 1973.

\bibitem{DuboisSalauenSalmon2003a}
F.~Dubois, M.~Sala\"un, and S.~Salmon.
\newblock First vorticity-velocity-pressure numerical scheme for the {S}tokes problem.
\newblock {\em Comput. Methods Appl. Mech. Engrg.}, 192(44-46):4877--4907, 2003.

\bibitem{DuboisSalauenSalmon2003}
F.~Dubois, M.~Sala\"un, and S.~Salmon.
\newblock Vorticity-velocity-pressure and stream function-vorticity formulations for the {S}tokes problem.
\newblock {\em J. Math. Pures Appl. (9)}, 82(11):1395--1451, 2003.

\bibitem{FuGuzmanNeilan2020}
G.~Fu, J.~Guzm\'{a}n, and M.~Neilan.
\newblock Exact smooth piecewise polynomial sequences on {A}lfeld splits.
\newblock {\em Math. Comp.}, 89(323):1059--1091, 2020.

\bibitem{GiraultRaviart1986}
V.~Girault and P.-A. Raviart.
\newblock {\em Finite element methods for {N}avier-{S}tokes equations: Theory and algorithms}, volume~5 of {\em Springer Series in Computational Mathematics}.
\newblock Springer-Verlag, Berlin, 1986.

\bibitem{GopalakrishnanLedererSchoeberl2020}
J.~Gopalakrishnan, P.~L. Lederer, and J.~Sch\"{o}berl.
\newblock A mass conserving mixed stress formulation for {S}tokes flow with weakly imposed stress symmetry.
\newblock {\em SIAM J. Numer. Anal.}, 58(1):706--732, 2020.

\bibitem{GopalakrishnanLedererSchoeberl2020a}
J.~Gopalakrishnan, P.~L. Lederer, and J.~Sch\"{o}berl.
\newblock A mass conserving mixed stress formulation for the {S}tokes equations.
\newblock {\em IMA J. Numer. Anal.}, 40(3):1838--1874, 2020.

\bibitem{GuzmanLischkeNeilan2020}
J.~Guzm\'{a}n, A.~Lischke, and M.~Neilan.
\newblock Exact sequences on {P}owell-{S}abin splits.
\newblock {\em Calcolo}, 57(2):Paper No. 13, 25, 2020.

\bibitem{GuzmanNeilan2014}
J.~Guzm\'an and M.~Neilan.
\newblock Conforming and divergence-free {S}tokes elements in three dimensions.
\newblock {\em IMA J. Numer. Anal.}, 34(4):1489--1508, 2014.

\bibitem{GuzmanNeilan2014a}
J.~Guzm\'an and M.~Neilan.
\newblock Conforming and divergence-free {S}tokes elements on general triangular meshes.
\newblock {\em Math. Comp.}, 83(285):15--36, 2014.

\bibitem{HarlowWelchothers1965Numerical}
F.~H. Harlow and J.~E. Welch.
\newblock Numerical calculation of time-dependent viscous incompressible flow of fluid with free surface.
\newblock {\em Physics of fluids}, 8(12):2182, 1965.

\bibitem{HuHuangLin2014}
J.~Hu, Y.~Huang, and Q.~Lin.
\newblock Lower bounds for eigenvalues of elliptic operators: by nonconforming finite element methods.
\newblock {\em J. Sci. Comput.}, 61(1):196--221, 2014.

\bibitem{HuMa2015}
J.~Hu and R.~Ma.
\newblock The enriched {C}rouzeix-{R}aviart elements are equivalent to the {R}aviart-{T}homas elements.
\newblock {\em J. Sci. Comput.}, 63(2):410--425, 2015.

\bibitem{HuZhangZhang2022}
K.~Hu, Q.~Zhang, and Z.~Zhang.
\newblock A family of finite element {S}tokes complexes in three dimensions.
\newblock {\em SIAM J. Numer. Anal.}, 60(1):222--243, 2022.

\bibitem{HuangTang2025}
X.~Huang and Z.~Tang.
\newblock Robust and optimal mixed methods for a fourth-order elliptic singular perturbation problem.
\newblock {\em J. Sci. Comput.}, 105(3):Paper No. 72, 29, 2025.

\bibitem{HuangWang2023}
X.~Huang and F.~Wang.
\newblock Analysis of divergence free conforming virtual elements for the {B}rinkman problem.
\newblock {\em Math. Models Methods Appl. Sci.}, 33(6):1245--1280, 2023.

\bibitem{JohnLinkeMerdonNeilanEtAl2017}
V.~John, A.~Linke, C.~Merdon, M.~Neilan, and L.~G. Rebholz.
\newblock On the divergence constraint in mixed finite element methods for incompressible flows.
\newblock {\em SIAM Rev.}, 59(3):492--544, 2017.

\bibitem{KimChungLee2013}
H.~H. Kim, E.~T. Chung, and C.~S. Lee.
\newblock A staggered discontinuous {G}alerkin method for the {S}tokes system.
\newblock {\em SIAM J. Numer. Anal.}, 51(6):3327--3350, 2013.

\bibitem{Linke2014}
A.~Linke.
\newblock On the role of the {H}elmholtz decomposition in mixed methods for incompressible flows and a new variational crime.
\newblock {\em Comput. Methods Appl. Mech. Engrg.}, 268:782--800, 2014.

\bibitem{MardalTaiWinther2002}
K.~A. Mardal, X.-C. Tai, and R.~Winther.
\newblock A robust finite element method for {D}arcy-{S}tokes flow.
\newblock {\em SIAM J. Numer. Anal.}, 40(5):1605--1631, 2002.

\bibitem{MazyaRossmann2010}
V.~Maz'ya and J.~Rossmann.
\newblock {\em Elliptic equations in polyhedral domains}, volume 162 of {\em Mathematical Surveys and Monographs}.
\newblock American Mathematical Society, Providence, RI, 2010.

\bibitem{Nedelec1980}
J.-C. N\'{e}d\'{e}lec.
\newblock Mixed finite elements in {${\bf R}\sp{3}$}.
\newblock {\em Numer. Math.}, 35(3):315--341, 1980.

\bibitem{Nedelec1986}
J.-C. N\'ed\'elec.
\newblock A new family of mixed finite elements in {${\bf R}^3$}.
\newblock {\em Numer. Math.}, 50(1):57--81, 1986.

\bibitem{RaviartThomas1977}
P.-A. Raviart and J.~M. Thomas.
\newblock A mixed finite element method for 2nd order elliptic problems.
\newblock In {\em Mathematical aspects of finite element methods ({P}roc. {C}onf., {C}onsiglio {N}az. delle {R}icerche ({C}.{N}.{R}.), {R}ome, 1975)}, volume Vol. 606 of {\em Lecture Notes in Math.}, pages 292--315. Springer, Berlin-New York, 1977.

\bibitem{ScottVogelius1985}
L.~R. Scott and M.~Vogelius.
\newblock Norm estimates for a maximal right inverse of the divergence operator in spaces of piecewise polynomials.
\newblock {\em RAIRO Mod\'el. Math. Anal. Num\'er.}, 19(1):111--143, 1985.

\bibitem{TaiWinther2006}
X.-C. Tai and R.~Winther.
\newblock A discrete de {R}ham complex with enhanced smoothness.
\newblock {\em Calcolo}, 43(4):287--306, 2006.

\bibitem{TaylorHood1973}
C.~Taylor and P.~Hood.
\newblock A numerical solution of the {Navier-Stokes} equations using the finite element technique.
\newblock {\em Computers. \& Fluids}, 1(1):73--100, 1973.

\bibitem{WangYe2007}
J.~Wang and X.~Ye.
\newblock New finite element methods in computational fluid dynamics by {$H(\rm div)$} elements.
\newblock {\em SIAM J. Numer. Anal.}, 45(3):1269--1286, 2007.

\bibitem{WeiHuangLi2021}
H.~Wei, X.~Huang, and A.~Li.
\newblock Piecewise divergence-free nonconforming virtual elements for {S}tokes problem in any dimensions.
\newblock {\em SIAM J. Numer. Anal.}, 59(3):1835--1856, 2021.

\bibitem{XieXuXue2008}
X.~Xie, J.~Xu, and G.~Xue.
\newblock Uniformly-stable finite element methods for {D}arcy-{S}tokes-{B}rinkman models.
\newblock {\em J. Comput. Math.}, 26(3):437--455, 2008.

\bibitem{Zhang2005}
S.~Zhang.
\newblock A new family of stable mixed finite elements for the 3{D} {S}tokes equations.
\newblock {\em Math. Comp.}, 74(250):543--554, 2005.

\bibitem{Zhang2011}
S.~Zhang.
\newblock Divergence-free finite elements on tetrahedral grids for {$k\geq 6$}.
\newblock {\em Math. Comp.}, 80(274):669--695, 2011.

\end{thebibliography}
\end{document}